\documentclass[10pt]{ijnam}
\def\currentvolume{1}
\def\currentissue{1}
\def\currentyear{2021}
\def\month{March}
\pagespan{1}{18}
\copyrightinfo{2021}{} 

\usepackage{amssymb,version}
\usepackage{url}
\usepackage{verbatim}
\usepackage[titletoc]{appendix}
\usepackage{bm}
\usepackage{amsmath}
\usepackage{amsthm}
\usepackage{enumerate}
\usepackage{cases}
\usepackage{amsfonts}
\usepackage{graphicx}
\usepackage{nicefrac}
\usepackage{longtable}
\usepackage{color}
\usepackage{graphicx, amssymb, subcaption,graphics}
\usepackage{CJK}
\newtheorem{Theorem}{\sc Theorem}

\newtheorem{lemma}[Theorem]{\sc Lemma}

\newtheorem{remark}[Theorem]{\sc Remark}

\allowdisplaybreaks[3]
\catcode`\@=11
\@addtoreset{equation}{section}
\renewcommand\theequation{\thesection.\arabic{equation}}
\usepackage{hyperref}
\begin{document}
	
	\title[Finite element methods for a Keller-Segel system]
	{ERROR ESTIMATES AND BLOW-UP ANALYSIS OF\\
		A FINITE-ELEMENT APPROXIMATION FOR THE\\
		parabolic-elliptic   Keller-Segel system}
	
	\author[w. Chen]{Wenbin Chen}
	\address{
		Shanghai Key Laboratory for Contemporary Applied Mathematics, 
		School of Mathematical Sciences,
		Fudan University,
		Shanghai, 200433, China
	}
	\email{wbchen@fudan.edu.cn}
	
	\author[ q. Liu]{Qianqian Liu}
	\address{
		School of Mathematical Sciences,
		Fudan University,
		Shanghai, 200433, China
	}
	\email{qianqianliu21@m.fudan.edu.cn}
	
	\author[J. Shen]{Jie Shen}
	\address{
		Department of Mathematics,
		Purdue University,
		West Lafayette, IN 47907, USA
	}
	\email{shen7@purdue.edu}

	
	
	\date{}
	
	
	
	\abstract{The Keller-Segel equations are widely used for describing chemotaxis in biology. Recently,
		a new fully discrete scheme for this model was proposed in \cite{2020Unconditionally},
		mass conservation, positivity and energy decay were proved for the proposed scheme,
		which are important properties of the original system. In this paper,
		we establish the error estimates of this scheme. Then, based on the error estimates,
		we derive the finite-time blowup of nonradial numerical solutions under some conditions on the mass and the moment of the initial data.}
	\maketitle
	\section{Introduction.}\label{Intro}
	Keller and Segel first proposed a nonlinear model in the 1970s to describe the effect of cell aggregation in \cite{1970Keller,1971Keller}.
	A simplified Keller-Segel model in 2-D is given by
	\begin{align}
	\frac{\partial u}{\partial t} =&\Delta u - \chi\nabla\cdot(u\nabla v),\quad x\in\Omega,\ t>0, \label{1.1}\\
	0=&\Delta v-v+\alpha u,\quad x\in\Omega,\ t>0,\label{1}
	\end{align}
	where $\Omega\subset \mathbb{R}^2$ is a bounded domain with smooth boundary $\partial\Omega$. The unknown $u=u(x,t)$ and $v=v(x,t)$
	represent the concentration of the organism and chemoattractant respectively. The parameters $\chi,\alpha$ are positive constants with $\chi$ being the sensitivity of chemotaxis. The model is supplemented with  initial conditions
	\begin{equation*}
	u(x,t=0)=u_0(x),\ v(x,t=0)=v_0(x),\ x\in\Omega,
	\end{equation*}
	and no flux boundary conditions
	\begin{equation*}
	\frac{\partial u}{\partial\bm{n}}-\chi u\frac{\partial v}{\partial\bm{n}}=0,\ \frac{\partial v}{\partial\bm{n}}=0,\quad x\in\partial\Omega,\ t>0,\label{1.3}
	\end{equation*}
	where $\bm{n}$ denotes the unit outward normal vector to the boundary $\partial\Omega$, $\partial /\partial\bm{n}$
	represents differentiation along $\bm{n}$ on $\partial\Omega$.
	
	A different version of the Keller-Segel model consists in replacing \eqref{1} by
	\begin{equation}
	\frac{\partial v}{\partial t}=\Delta v-v+\alpha u,\quad x\in\Omega,\ t>0.\label{1.2}
	\end{equation}

	The equation \eqref{1.1} describes the motion of the organism $u$. The term $F=-\nabla u+\chi u\nabla v$ is the flux, and
	the effect of diffusion $-\Delta u$ and that of chemotaxis $\chi\nabla\cdot(u\nabla v)$ are competing for $u$ to vary.
	The equation \eqref{1} describes the change in concentration of the chemoattractant $v$,
	it is influenced by the diffusion and the decay of the chemoattractant as well as the growth of the organism.
	In general, the chemoattractant particles are much smaller than the organism particles, thus it diffuses faster,
	which means that the diffusion of the chemoattractant will reach the equilibrium state in a relatively short time.
	The model \eqref{1.1}-\eqref{1} is called parabolic–elliptic system.
	On the other hand, \eqref{1.1} with \eqref{1.2} is a parabolic–parabolic system.

	The solution of the Keller-Segel model \eqref{1.1}-\eqref{1} has several well-known properties,  particularly,
	it may blow up in finite time. Various aspects
	and results for the classical Keller-Segel model since 1970, along with some
	open questions,  are summarized in \cite{2004Horstmann}. Positivity, mass conservation and energy dissipation of Keller-Segel equations can be found in \cite{2000Behavior},\cite{2001Concentration},\cite{2005Free},\cite{Hideo2008Local} and \cite{2019Calvez}, which plays an important role to study the Keller-Segel system. 
	Blanchet, Dolbeault and Perthame presented   in \cite{2006Blanchet}  a detail proof of the existence of weak solutions when the initial mass is below the critical mass,
	above which any solution to the parabolic-elliptic systems blows up in finite time in the whole Euclidean space.
	In \cite{2001Blowup}, Nagai demonstrated the finite-time blowup of nonradial solutions  under some assumptions on the mass and the moment of the initial data.
	As for the parabolic-parabolic systems, Blanchet proved  in \cite{2012Blanchet} the optimal critical mass of the solutions in $\mathbb{R}^d$ with $d\ge3$. Wei proved that for every nonnegative initial data in $L^1(R^2)$, the 2-D Keller-Segel equation is globally well-posed if and only if the total mass $M\le 8\pi$ in \cite{wei2018}.

	Although the large time behavior of the solution of the Keller-Segel model \eqref{1.1}-\eqref{1} has been well studied,
	there is still much to explore on the numerical side. Since the Keller-Segel equations possess three important properties:
	positivity, mass conservation and energy dissipation, it is preferable that  numerical schemes can  preserve these properties at the discrete level.
	In \cite{2019BLOW}, the existence of weak solutions and upper bounds for the blow-up time for time-discrete
	(including the implicit Euler, BDF and Runge-Kutta methods) approximations of the parabolic-elliptic Keller-Segel models in the two-dimensional
	whole space are established. Liu, Li and Zhou proposed a numerical method in \cite{2016liu} which preserves both positivity
	and asymptotic limit, the proposed numerical method does not generate negative density
	if initialized properly under a less strict stability condition. Saito and Suzuki presented a finite difference scheme  in \cite{2005SaitoNotes} which satisfies the conservation of a discrete $L^1$ norm.
	
	Some finite element methods are proposed in previous works. Saito presented a finite element scheme for parabolic-elliptic systems in \cite{2007NorikazuConservative} that satisfies  both positivity and mass conservation properties. Under some assumptions on the regularity of solutions, the error estimates were established. Saito further constructed the finite element methods to the parabolic-parabolic systems in \cite{2011SaitoERRORAO} and derived error analysis by using analytical semigroup theory. Gurusamy and Balachandran proposed a finite element method for parabolic-parabolic systems and established the existence of approximate solutions by using Schauder's fixed point theorem in \cite{2018gurusamy}. Further the error estimates for the approximate solutions in $H^1$-norm were derived.

	The discontinuous Galerkin methods can be also used to solve the Keller-Segel equations. Epshteyn and Kurganov developed a family of new interior penalty discontinuous Galerkin methods and proved error estimates for the proposed high-order discontinuous Galerkin methods in \cite{2008epshteyn}. Epshteyn and Izmirlioglu further constructed a discontinuous Galerkin method for Keller-Segel model in \cite{2009epshteyn} and obtained fully discrete error estimates for the proposed scheme. 
	In 2017, Li, Shu and Yang applied the local discontinuous Galerkin (LDG) method to 2D Keller–Segel  chemotaxis model in \cite{2017li}, they improved the results upon \cite{2008epshteyn} and gave optimal rate of convergence under
	special finite element spaces before the blow-up occurs. In 2019, Guo, Li and Yang constructed a consistent numerical energy and prove the energy dissipation with the LDG discretization in \cite{li2019}. 
	
	Another important numerical methods for Keller-Segel models are finite volume methods since the positivity property can be naturally preserved.
	Filbet proposed  in \cite{2006Filbet} a finite volume scheme for the parabolic-elliptic system, and
	by assuming the CFL condition $\chi\Delta t\mathcal{D}_{\mathcal{T},1}<1$ and the initial datum $n^0\ge a^0>0$,
	he proved existence and uniqueness of the numerical solution by using the Browder fixed point theorem,
	and showed that the numerical approximation converges to the exact solution under some assumptions. 
	In 2016, Zhou and Saito
	proposed a finite volume scheme  in \cite{2016SaitoFinite}, and established error estimates in $L^p$ norm with a suitable $p>2$ for the two dimensional case
	under some regularity assumptions of solutions and admissible mesh.
	By focusing on the radially symmetirc solution,
	they derived some {\it a prior} estimates to study the blow-up phenomenon of numerical solution. 
	
    There have been growing interests in positivity-preserving analysis for gradient flows with logarithmic energy potential. 
	Some theoretical analysis of the positivity-preserving property and the energy stability have been explored for these numerical schemes for certain systems, such as Cahn-Hilliard systems in \cite{2019chen,2019dong,2020dong,2021yuan,2021dong}, the Poisson-Nernst-Planck-Cahn-Hilliard systems in \cite{2021qian}, the Poisson-Nernst-Planck systems in \cite{2021wangcheng}, the thin film model without slope selection in \cite{2018liwj} and a structure-preserving, operator splitting scheme for reaction-diffusion systems in \cite{2021liu}. The techniques of the higher order consistency analysis combined with rough error estimate and refined one have been presented in \cite{2021wangcheng,2020duan,2021duan} which will be utilized in the following to obtain the convergence analysis. 
	
	Recently, a new approach for constructing positivity preserving schemes was proposed in \cite{2020Unconditionally}.
	The key for this approach is to write  $\Delta u$ as $\nabla\cdot(u\nabla\log u)$ in \eqref{1.1}, and then use a convex splitting idea to construct
	mass conservative, bound preserving, and uniquely solvable schemes for \eqref{1.1}-\eqref{1} and  for  \eqref{1.1}-\eqref{1.2}.
	The main purposes of this  paper are to
	establish the convergence of the fully discrete scheme proposed in \cite{2020Unconditionally},
	and to show  the finite-time blowup of numerical solutions under some conditions on the mass and moments of the initial data.
	More precisely,  let $u_h^k$ be an approximation of $u(\cdot,k\tau)$,
	where $\tau>0$ is the time step and $k\in\mathbb{N}$.
	Let $\theta=(u_0,1)$ be the initial mass, and $M^k=(u_h^k,\varphi)_h$ be the moment of $u_h^k$.
	Our first goal is to establish the error estimates for the fully discrete scheme proposed in \cite{2020Unconditionally} (cf. Theorem \ref{th.err}).
	Another important feature of the Keller-Segel system \eqref{1.1}-\eqref{1} is  that the solution may blow up in finite time under certain conditions on the initial data.
	Our second goal is to show that  the numerical solution will also  blow up in finite time  under similar conditions on the initial data (cf. Theorem \ref{th:blow}). 
	Many previous works (see \cite{2005SaitoNotes,2007NorikazuConservative,2020Unconditionally}) show that the numerical solution seems to blow up under large initial data by several numerical experiments. However, there is still much to explore on the theoretical proof of blowup phenomenon besides the radial numerical solution in \cite{2016SaitoFinite} mentioned before.

	The rest of the paper is organized as follows. In Section \ref{sec2}, we recall some properties of the classical Keller-Segel equations, including
	its finite-time blowup behavior.  In Section \ref{sec3}, we introduce the fully discrete scheme constructed in  \cite{2020Unconditionally} and carry out a rigorous  error analysis.
	In Section \ref{sec4},  we show that the numerical solution will blow up in finite time under suitable conditions on the initial data.

	\section{The Keller-Segel equations}\label{sec2}
	In this section, we recall some properties for the Keller-Segel system \eqref{1.1}-\eqref{1} with no flux
	boundary conditions. In addition, we
	assume the initial value $u_0\in  W^{2,p}(\Omega),\ 1<p<\infty$, and satisfies
	\begin{equation*}
	u_0\ge0\quad \text{and}\quad u_0\not\equiv0,\quad \forall x\in\Omega.
	\end{equation*}
	It was shown in \cite{2000Behavior} that there exist some $T>0$ such that \eqref{1.1}-\eqref{1} is well posed in the time interval $[0,T]$.
	Moreover, it holds that
	\begin{Theorem}
		The Keller-Segel system \eqref{1.1}-\eqref{1} satisfies the following properties:
		\begin{enumerate}[(i)]
			\item Positivity preserving:
			\begin{equation*}
			u(x,t)>0,\quad v(x,t)>0 \quad\text{on} \quad\Omega\times(0,T]. \\
			\end{equation*}
			In fact, it is a consequence of the strong maximum principle \cite{1995Diaz}\ .
			\item Mass conservation:
			\begin{equation*}
			\int_{\Omega}u(x,t)\text{d}x=\int_{\Omega}u_0(x)\text{d}x, \quad\text{for all}\quad t>0. \\
			\end{equation*}
			It is immediately follows from
			\begin{equation*}
			\dfrac{\text{d}}{\text{d}t}\int_{\Omega}u(x,t)\text{d}x=\int_{\partial \Omega}
			\left(\dfrac{\partial u}{\partial \bm{n}}-u\dfrac{\partial v}{\partial \bm{n}}\right)\text{d}S=0\ .
			\end{equation*}
			As a consequence of (i) and (ii), we obtain the conservation of $L^1$ norm, namely
			$$\|u(t)\|_{L^1(\Omega)}=\|u_0\|_{L^1(\Omega)}.$$
			\item  Energy decay:
			\begin{equation*}
			\dfrac{\text{d}F[u(t),v(t)]}{\text{d}t}=-\int_{\Omega}u|\nabla\cdot(\log u-\chi v)|^2\text{d}x\le0,
			\end{equation*}
			where the free energy of \eqref{1.1}-\eqref{1} is defined by $$F[u,v]=
			\int_{\Omega}\left(u(\log u-1)-\chi uv+\dfrac{\chi}{2\alpha}|\nabla v|^2+\dfrac{\chi}{2\alpha}v^2\right)\text{d}x\ .$$
		\end{enumerate}
	\end{Theorem}
	The following result is shown in \cite{2001Blowup}.
	\begin{lemma}\cite{2001Blowup}\label{lem:blowup}
		Let $q\in\Omega$ and $0<r_1<r_2<dist(q,\partial\Omega)$, where $dist(q,\partial\Omega)$ is the distance between $q$ and $\partial\Omega$.
		Then there exist positive constants $C_1,C_2$ depending only on $r_1,r_2$ and $dist(q,\partial\Omega)$ such that for $t\in(0,T]$,
		\begin{equation}
		\begin{aligned}
		&\dfrac{\text{d}}{\text{d}t}\int_{\Omega}u(x,t)\Phi(x)\text{d}x\\
		\le\ &4\int_{\Omega}u_0(x)\text{d}x-\dfrac{\alpha\chi}{2\pi}\left(\int_{\Omega}u_0(x)\text{d}x\right)^2
		+C_1\left(\int_{\Omega}u_0(x)\text{d}x\right)\left(\int_{\Omega}u(x,t)\Phi(x)\text{d}x\right)\\
		&+C_2\left(\int_{\Omega}u_0(x)\text{d}x\right)^{3/2}\left(\int_{\Omega}u(x,t)\Phi(x)\text{d}x\right)^{1/2},\\
		\end{aligned}
		\end{equation}
		where $\Phi(x)=\phi(|x-q|)$ with
		\begin{eqnarray*}
			\phi(r)=
			\begin{cases}
				r^2 &if\quad 0\le r\le r_1,\\
				a_1r^2+a_2r+a_3 &if\quad r_1<r\le r_2,\\
				r_1r_2 &if\quad r>r_2,\\
			\end{cases}
		\end{eqnarray*}
		where $a_1=-\dfrac{r_1}{r_2-r_1},\quad a_2=\dfrac{2r_1r_2}{r_2-r_1},\quad a_3=-\dfrac{r_1^2r_2}{r_2-r_1}$.
	\end{lemma}
	The finite-time blowup behavior is then proved using the above result.
	\begin{Theorem}\cite{2001Blowup} \label{thm:blowup}
		Assume that $\int_{\Omega}u_0(x)\text{d}x>8\pi/(\alpha\chi)$, if $\int_{\Omega}u_0(x)|x-q|^2\text{d}x$ is sufficiently small,
		then the solution $(u,v)$ to \eqref{1.1}-\eqref{1} blows up in finite time.
	\end{Theorem}
	Moreover,  the following pointwise estimates  for $v$ is established in \cite{2015Fujie}. An application of the Neumann semigroup leads to
	\begin{equation*}
	\begin{aligned}
	v(x,t)&\ge\left(\int_{0}^{\infty}\dfrac{1}{(4\pi t)^{\frac{n}{2}}}e^{-(t+\frac{(\mathrm{diam}\Omega)^2}{4t})}\mathrm{d}t\right)\int_{\Omega}u(x,t)\mathrm{d}x\\
	&=\|u_0\|_{L^1(\Omega)}\int_{0}^{\infty}\dfrac{1}{(4\pi t)^{\frac{n}{2}}}e^{-(t+\frac{(\mathrm{diam}\Omega)^2}{4t})}\mathrm{d}t,
	\end{aligned}
	\end{equation*}
	for all $x\in\Omega,\ t\in(0,T]$, whenever $(u,v)$ solves \eqref{1.1}-\eqref{1} in $\Omega\times(0,T]$ for some $T>0$.

	In this paper, we assume that
	\begin{equation*}
	u(x,t)\ge\epsilon_0 \quad \text{for some } \epsilon_0>0,\quad (x,t)\in\overline{\Omega}\times(0,T].
	\end{equation*}

	\section{The fully discrete scheme and error estimates}\label{sec3}
	In this section, we describe the fully discrete scheme  in \cite{2020Unconditionally} for \eqref{1.1}-\eqref{1},
	construct the error equations and establish the error estimates.
	
	We now give a precise description of our  finite element space $X_h$.
	Given a  triangulation $\mathcal{T}$ for $\Omega$, we let $Z_h$ consists of all the vertices excluding those where   Dirichlet boundary conditions are prescribed.
	We define $X_h$ to be the finite element space spanned by the piecewise linear continuous functions based on  $\mathcal{T}$.
	Let $e$ be a triangle of the triangulation $\mathcal{T}$, and $P_{e,i},i=1,2,3$ be its vertices, we define the quadrature formula
	$$Q_{e}(f)=\dfrac{1}{3}\text{area}(e)\sum_{i=1}^{3}f(P_{e,i})\approx\int_{e}f\mathrm{d}x.$$
	We recall that  \cite{2006Thomee}
	\begin{equation*}
	\Big|Q_{e}(f)-\int_{e}f\mathrm{d}x\Big|\le Ch^2\sum_{|\alpha|=2}\|D^{\alpha}f\|_{L^1(e)}.
	\end{equation*}
	We then define the discrete inner product in $X_h$ by
	$$(u,v)_h=\sum_{e\in\mathcal{T}}Q_{e}(uv),$$
	the corresponding norm is defined by $\|\cdot\|_{L_h^2}$.
	We have the following estimates in \cite{2006Thomee} for the quadrature error.
	\begin{lemma}\label{lem:qua}
		Let $\epsilon_h(\cdot,\cdot)=(\cdot,\cdot)_h-(\cdot,\cdot)$ denote the quadrature error, then we have
		\begin{equation*}
		|\epsilon_h(u,v)|\le Ch^2\|\nabla u\|_{L^2(\Omega)}\|\nabla v\|_{L^2(\Omega)},\ \forall\ u,v\in X_h.
		\end{equation*}
	\end{lemma}
    Applying Lemma \ref{lem:qua},  the norm $\|\cdot\|_{L_h^2}$ has the following property 
    \begin{align}\label{norm}
    \|\nabla u\|_{L_h^2(\Omega)}=\|\nabla u\|_{L^2(\Omega)},\ \forall\ u\in X_h.
    \end{align}

	Let $I_h:\; C(\Omega)\rightarrow X_h$ be the
	Lagrange interpolation operator, which has the approximation property \cite{2008brenner} that  for all $g\in H^2(\Omega)\cap H_0^1(\Omega)$,
	\begin{equation}\label{eq:I_h}
	\|I_hg-g\|_{L^2(\Omega)}\le Ch^2\|g\|_{H^2(\Omega)}\quad \text{and}\quad \|\nabla(I_hg-g)\|_{L^2(\Omega)}\le Ch\|g\|_{H^2(\Omega)}.
	\end{equation}
	
	The fully discrete scheme proposed in \cite{2020Unconditionally} for \eqref{1.1}-\eqref{1} is to find $(u_h^{k+1},v_h^{k+1})\in X_h\times X_h$ such that for all $ (\varphi,\psi)\in X_h\times X_h$, 
	\begin{align}
	&		\left(\overline{\partial} u_h^k,\varphi\right)_h\ +\left(u_h^k\nabla \left(I_h \log u_h^{k+1}\right),\nabla \varphi\right)_h-\chi\left(u_h^k\nabla v_h^k,\nabla \varphi\right)_h=0,  \label{eq:fem1}\\
	&	(\nabla v_h^{k+1},\nabla \psi)+(v_h^{k+1},\psi)-\alpha(u_h^{k+1},\psi)_h=0, \label{eq:fem2}
	\end{align}
with  the initial value  $u_h^0:=I_hu_0$. 	Here,  the $(\cdot,\cdot)$ represents the usual $L^2$ inner product, the $(\cdot,\cdot)_h$ is the discrete inner product defined above, and $\overline{\partial} u_h^k$ is the forward Euler difference quotient approximating to $\partial_tu(t^k)$ defined by
	$$\overline{\partial} u_h^k=\frac{u_h^{k+1}-u_h^k}{\tau}.$$
	In this setting, the authors in \cite{2020Unconditionally} proved the following.
	\begin{lemma}\cite{2020Unconditionally}
		The numerical scheme \eqref{eq:fem1}-\eqref{eq:fem2} has the following properties:
		\begin{enumerate}[(i)]
			\item Unique solvability:\\
			The scheme \eqref{eq:fem1}-\eqref{eq:fem2} has a unique solution $(u_h^{k+1},v_h^{k+1})\in X_h\times X_h$.
			\item Positivity preserving:
				If $u_h^k>0 $, then $u_h^{k+1}>0$.
			\item Mass conservation:	
			$$(u_h^{k+1},1)_h=(u_h^k,1)_h=(u_h^0,1)_h,$$
			It is immediately derived by taking $\varphi =1$ in \eqref{eq:fem1}.
			\item  Energy decay:
			\begin{equation*}
			\tilde{F}^{k+1}-\tilde{F}^{k}\le-\tau\left(u_h^k\nabla(I_h(\log u_h^{k+1}-\chi v_h^k)),\nabla(I_h(\log u_h^{k+1}-\chi v_h^k))\right)_h\le0,
			\end{equation*}
			where the discrete energy of \eqref{eq:fem1}-\eqref{eq:fem2} is defined by
			$$\tilde{F}^{k}[u,v]=(u_h^k(\log u_h^k-1),1)_h-\chi(u_h^k,v_h^k)_h+\dfrac{\chi}{2\alpha}\|\nabla v_h^k\|^2+\dfrac{\chi}{2\alpha}\|v_h^k\|^2.$$
		\end{enumerate}
		
	\end{lemma}

	We denote by $(u,v)$ the exact solution pair to the original equations \eqref{1.1}-\eqref{1.2}, and all the upper bounds for the exact solution are denoted as $C$. 
	We set  $u^k=u(t_k),\, v^k=v(t_k)$, and denote
	\begin{equation*}
	e_u^k:=u^k-u_h^k,\quad e_v^k:=v^k-v_h^k,\quad \forall k\in\mathbb{N}.
	\end{equation*}
	The following theorem is the main result of this section.
	\begin{Theorem}\label{th.err}
		Assume $u_0\in W^{2,p}(\Omega)(1<p<\infty)$ and the exact solution pair $(u,v)$ is smooth enough for a fixed final time $T>0$.
		Then, provided $\tau$ and $h$ are sufficiently small and under the mild mesh-sizes requirement
		$\tau\le Ch$,
		we have the following error estimates
		\begin{equation*}
		\|e_u^m\|_{L_h^2(\Omega)}+	\|e_v^m\|_{H_h^1(\Omega)}+\Big(\tau\sum_{k=0}^{m-1}\|\nabla e_u^{k+1}\|_{L_h^2(\Omega)}^2\Big)^{1/2}\le C(\tau+h),\quad \forall m\in\mathbb{N},
		\end{equation*}
		where $t_m=m\tau\le T$, $C>0$ is independent of $\tau$ and $h$.
	\end{Theorem}
	
	The proof for this theorem will be carried out with a sequence of procedures that we describe below.
	\begin{remark}
	The mesh-sizes requirement $\tau\le Ch$ in Theorem \ref{th.err} is proposed to obtain a higher order consistency analysis via a perturbation argument, which is needed to get the separation property and the $W^{1,\infty}$ bound for the numerical solution.
	\end{remark}

	\subsection{Higher order consistent approximation to \eqref{eq:fem1}-\eqref{eq:fem2}}
	In this subsection, we apply the perturbation argument method in \cite{2021wangcheng} to the finite element scheme to construct $f_{1},f_{2},f_{3}$ such that
	\begin{equation*}
	\hat{u}:=u+hf_{1}+h^2f_{2}+h^3f_{3},
	\end{equation*}
	is  consistent  with the given numerical scheme \eqref{eq:fem1}-\eqref{eq:fem2} at the order  $O(h^4)$. 
	The following lemma is used to construct $f_{1},f_{2},f_{3}$ and the proof is given in Appendix.
	
By applying a perturbation argument, a higher order $O(h^4)$ consistency is satisfied for $\hat{u}$, which is needed to obtain the separation property and a $W^{1,\infty}$ bound for the numerical solution.
	\begin{lemma}\label{lem:f3}
	Suppose that $\tau\le Ch$ and $u$ is smooth enough, then there exist bounded smooth  functions $f_1,f_2,f_3$, such that $\hat{u}=u+hf_{1}+h^2f_{2}+h^3f_{3}$ satisfies
	\begin{align}\label{eq:consistency}
	\left(\overline{\partial}\hat{u}^k,\varphi\right)_h+\left(\hat{u}^k\nabla I_h\log \hat{u}^{k+1},\nabla\varphi\right)_h-\chi\left(\hat{u}^k\nabla A_h\hat{u}^k,\nabla\varphi\right)_h=\langle\mathcal{R}^k(\hat{u}),\varphi\rangle,
	\end{align}
	for all $\varphi\in X_h$, $k\in\mathbb{N}$, where $A_h=\alpha(-\Delta_h+I)^{-1}Q_h$ and $\langle \cdot,\cdot\rangle$ denotes the duality product satisfying
	\begin{align}\label{eq:r}
		|\langle\mathcal{R}^k(\hat{u}^k),\varphi\rangle|\le Ch^4\|\varphi\|_{H^1},
	\end{align}
	where $C$ depends on the regularity of the solution $u$.
\end{lemma}  
	\begin{remark}
		Under the conditions that
		the exact solution $u\ge\epsilon_0$ for some $\epsilon_0>0$,
		  and $h$ is sufficiently small, we obtain that
		\begin{equation}\label{eq:u}
		\hat{u}\ge\dfrac{\epsilon_0}{2}.
		\end{equation}
		Since the correction functions $f_j,j=1,2,3$ only depend on the exact solution $u$, they are bounded in $W^{1,\infty}$ norm.
		Then, we can obtain the following $W^{1,\infty}$ bound for $\hat{u}$:
		\begin{equation}\label{eq:w1,infty}
		\|\hat{u}^k\|_{W^{1,\infty}}\le C,\quad\forall k\ge0.
		\end{equation}
	\end{remark}
	\subsection{A rough error estimate}
	In this subsection, we derive the strict separation property and a uniform $W^{1,\infty}$ bound for the numerical solution. 
	
		We recall the following inverse estimate in \cite[p.111, Lemma 4.5.3]{2008brenner}.
	\begin{lemma}
		Given a quasi-uniform triangulation $\mathcal{T}$ on domain $\Omega\subset \mathbb{R}^n$, and $X_h$ be a finite-dimensional subspace of $W^{l,p}(K)\cap W^{m,q}(K)$, where $1\le p\le\infty,1\le q\le\infty$ and $0\le m\le l$. Then there exists a positive constant $C$ such that for all $u\in X_h$, we have
		\begin{equation*}
		\|u\|_{W^{l,p}(K)}\le Ch^{m-l+n/p-n/q}\|u\|_{W^{m,q}(K)},
		\end{equation*}
		where $C$ is independent of $u$.
	\end{lemma}
	We will also use the following discrete Gronwall inequality in \cite{1990shen,Yan2018ASE}.
	\begin{lemma}\label{lem:gronwall}
		Assume that $\tau>0,B>0,\{a_k\},\{b_k\},\{\gamma_k\}$ are non-negative sequences such that \begin{equation*}
		a_m+\tau\sum_{k=1}^{m}b_k\le\tau\sum_{k=1}^{m-1}\gamma_ka_k+B,\ m\ge 1.
		\end{equation*}
		Then
		\begin{equation*}
		a_m+\tau\sum_{k=1}^{m}b_k\le B\mathrm{exp}\left(\tau\sum_{k=1}^{m-1}\gamma_k\right),\ m\ge 1.
		\end{equation*}
	\end{lemma}

Define an alternative error function:
$$\tilde{u}^k:=I_h\hat{u}^k-u_h^k,\quad\forall k\in\mathbb{N}.$$
	Subtracting the numerical scheme \eqref{eq:fem1} from the consistency estimate \eqref{eq:consistency} implies that
	\begin{equation}\label{eq:rough}	 (\overline{\partial}\tilde{u}^k,\varphi)_h=-(\tilde{u}^k\nabla\mathcal{V}_u^k+u_h^k\nabla\tilde{\mu}_u^k,\nabla\varphi)_h+\langle\mathcal{R}^k,\varphi\rangle,
	\end{equation}
	where
	\begin{equation*}
	\mathcal{V}_u^k:=I_h\log \hat{u}^{k+1}-\chi A_h\hat{u}^k,
	\end{equation*}
	\begin{equation*}
	\tilde{\mu}_u^k:=I_h\log \hat{u}^{k+1}-I_h\log u_h^{k+1}-\chi A_h\tilde{u}^k.
	\end{equation*}
	Since $\mathcal{V}_u^k$ only depends on the exact solution, we can assume
	\begin{equation}\label{eq:v}
	\|\mathcal{V}_u^k\|_{W^{2,\infty}}\le C.
	\end{equation}
	\begin{lemma}\label{w1pbound}
		The numerical solutions of the scheme \eqref{eq:fem1}-\eqref{eq:fem2} have the strict separation property and a uniform $W^{1,\infty}$ bound:
		\begin{align*}
		u_h^{k}\ge\dfrac{\epsilon_0}{4},\ \|u_h^{k}\|_{W^{1,\infty}}\le C^*,
		\end{align*}
		for all $0\le k\le T/\tau$, where $\epsilon_0$ and $C^*$ are positive constants.
	\end{lemma}
	\begin{proof}
		We shall first make the following assumption at the previous time step:
		\begin{equation}\label{eq:tilde u}
		\|\tilde{u}^k\|_{L_h^2}\le Ch^{15/4}.
		\end{equation}
		Then, we will demonstrate that such an assumption will be recovered at the next time step in Section \ref{recovery}.
		
		Using the inverse inequality, we obtain a $W^{1,\infty}$ bound for the numerical error function:
		\begin{equation}\label{eq:11/4-1}
		\|\tilde{u}^k\|_{\infty}\le\dfrac{C\|\tilde{u}^k\|_{L_h^2}}{h}\le\dfrac{Ch^{15/4}}{h}\le Ch^{11/4}\le 1,
		\end{equation}
		\begin{equation*}
		\|\nabla\tilde{u}^k\|_{\infty}\le\dfrac{C\|\tilde{u}^k\|_{\infty}}{h}\le\dfrac{Ch^{11/4}}{h}\le Ch^{7/4}\le 1.
		\end{equation*}
		A combination of the above with \eqref{eq:w1,infty}, we get a $W^{1,\infty}$ bound for $u_h^k$ at the previous time step:
		\begin{equation*}
		\|u_h^k\|_{\infty}\le\|\hat{u}^k\|_{\infty}+\|\tilde{u}^k\|_{\infty}\le C+1\le C^*,
		\end{equation*}
		\begin{equation*}
		\|\nabla u_h^k\|_{\infty}\le\|\nabla\hat{u}^k\|_{\infty}+\|\nabla\tilde{u}^k\|_{\infty}\le C+1\le C^*.
		\end{equation*}
		Because of \eqref{eq:11/4-1}, taking  $h$ sufficiently small, we have
		\begin{equation*}
		\|\tilde{u}^k\|_{\infty}\le Ch^{11/4}\le \dfrac{\epsilon_0}{4}.
		\end{equation*}
		Then the strict separation property is valid for $u_h^k$:
		\begin{equation}\label{eq:u_h^k-separation}
		u_h^k\ge I_h\hat{u}^k-\|\tilde{u}^k\|_{\infty}\ge\dfrac{\epsilon_0}{4}.
		\end{equation}
		
		Taking $\varphi=\tau\tilde{\mu}_u^k$ in \eqref{eq:rough} leads to
		\begin{equation}\label{eq:error1}
		(\tilde{u}^{k+1},\tilde{\mu}_u^k)_h+\tau(u_h^k\nabla\tilde{\mu}_u^k,\nabla\tilde{\mu}_u^k)_h
		=(\tilde{u}^k,\tilde{\mu}_u^k)_h-\tau(\tilde{u}^k\nabla\mathcal{V}_u^k,\nabla\tilde{\mu}_u^k)_h+\tau\langle\mathcal{R}^k,\tilde{\mu}_u^k\rangle.
		\end{equation}
		Now we deal with the left hand side of \eqref{eq:error1}. To proceed the first term on the left hand side of \eqref{eq:error1}, notice that $((\tilde{u}^{k+1})^3,1)_h=0$ since $(\tilde{u}^{k+1},1)_h=(I_h\hat{u}^{0}-u_h^0,1)_h=0$ stated in Remark \ref{remak:massconserve}, using the H\"older inequality, we have
		\begin{equation}\label{eq:left1}
		\begin{aligned}
		&(\tilde{u}^{k+1},\tilde{\mu}_u^k)_h=(\tilde{u}^{k+1},I_h\log \hat{u}^{k+1}-I_h\log u_h^{k+1})_h+(\tilde{u}^{k+1},-\chi A_h\tilde{u}^k)_h\\
		\ge& \Big(\tilde{u}^{k+1},\dfrac{1}{\hat{u}^{k+1}}\tilde{u}^{k+1}\Big)_h+\frac{1}{2\|\zeta\|_{\infty}^2}\Big((\tilde{u}^{k+1})^3,1\Big)_h-C'\|\tilde{u}^{k+1}\|_{L_h^2}\|\tilde{u}^k\|_{L_h^2}\\
		\ge&\dfrac{1}{C}\|\tilde{u}^{k+1}\|_{L_h^2}^2-\dfrac{1}{C}\|\tilde{u}^{k+1}\|_{L_h^2}^2-C''\|\tilde{u}^k\|_{L_h^2}^2\\
		=&-C''\|\tilde{u}^k\|_{L_h^2}^2,
		\end{aligned}
		\end{equation}
		where $\zeta$ lies between $\hat{u}^{k+1}$ and $u_h^{k+1}$,  and \eqref{eq:w1,infty} has been utilized in the second inequality. As for the second term on the left hand side of \eqref{eq:error1}, using the strict separation property of the numerical solution \eqref{eq:u_h^k-separation}, we have
		\begin{equation}\label{eq:left2}
		(u_h^k\nabla\tilde{\mu}_u^k,\nabla\tilde{\mu}_u^k)_h\ge \dfrac{\epsilon_0}{4}\|\nabla\tilde{\mu}_u^k\|_{L_h^2}^2.
		\end{equation}
		Next, we deal with the right hand side of \eqref{eq:error1}. We apply the H\"older inequality and the Young inequality:
		\begin{equation}\label{eq:right1}
		\begin{aligned}
		|(\tilde{u}^k,\tilde{\mu}_u^k)_h|&\le\|\tilde{u}^k\|_{-1}\|\nabla \tilde{\mu}_u^k\|_{L_h^2}
		\le C\|\tilde{u}^k\|_{L_h^2}\|\nabla\tilde{\mu}_u^k\|_{L_h^2}\\
		&\le \dfrac{6C}{\epsilon_0\tau}\|\tilde{u}^k\|_{L_h^2}^2
		+\dfrac{\epsilon_0\tau}{24}\|\nabla\tilde{\mu}_u^k\|_{L_h^2}^2.
		\end{aligned}
		\end{equation}
		An application of the Cauchy-Schwarz inequality and \eqref{eq:v} leads to
		\begin{equation*}
		\begin{aligned}\label{eq:right2}
		|-(\tilde{u}^k\nabla\mathcal{V}_u^k,\nabla\tilde{\mu}_u^k)_h|&
		\le\|\tilde{u}^k\nabla\mathcal{V}_u^k\|_{L_h^2}\|\nabla\tilde{\mu}_u^k\|_{L_h^2}
		\le\|\nabla\mathcal{V}_u^k\|_{\infty}\|\tilde{u}^k\|_{L_h^2}\|\nabla\tilde{\mu}_u^k\|_{L_h^2}\\
		&\le\dfrac{6C}{\epsilon_0}\|\tilde{u}^k\|_{L_h^2}^2+\dfrac{\epsilon_0}{24}\|\nabla\tilde{\mu}_u^k\|_{L_h^2}^2.
		\end{aligned}
		\end{equation*}
		Using the inequality \eqref{eq:r}, we have
		\begin{equation}\label{eq:right3}
		|\langle\mathcal{R}^k,\tilde{\mu}_u^k\rangle|\le Ch^4\|\nabla\tilde{\mu}_u^k\|_2
		\le\dfrac{6C}{\epsilon_0}h^8+\dfrac{\epsilon_0}{24}\|\nabla\tilde{\mu}_u^k\|_{L_h^2}^2.
		\end{equation}
		Substitution of \eqref{eq:left1}-\eqref{eq:right3} into \eqref{eq:error1} leads to
		\begin{equation*}
		\dfrac{\epsilon_0}{4}\tau\|\nabla\tilde{\mu}_u^k\|_{L_h^2}^2-C''\|\tilde{u}^k\|_{L_h^2}^2\le\dfrac{6C}{\epsilon_0\tau}\|\tilde{u}^k\|_{L_h^2}^2 +\dfrac{6C\tau}{\epsilon_0}\|\tilde{u}^k\|_{L_h^2}^2+\dfrac{6C}{\epsilon_0}h^8\tau+\dfrac{\epsilon_0\tau}{8}\|\nabla\tilde{\mu}_u^k\|_{L_h^2}^2.
		\end{equation*}
		Then we have the following estimate by applying $\tau\le Ch$
		\begin{equation}\label{eq:mu}
		\tau\|\nabla\tilde{\mu}_u^k\|_{L_h^2}\le Ch^{15/4}.
		\end{equation}
		
		Again, taking $\varphi=\tilde{u}^{k+1}-\tilde{u}^k$ in the error equation \eqref{eq:rough} leads to
		\begin{equation}\label{eq:error2}
		\|\tilde{u}^{k+1}-\tilde{u}^{k}\|_{L_h^2}\le\tau\left(\|\nabla\cdot(\tilde{u}^k\nabla\mathcal{V}_u^k)\|_{L_h^2}
		+\|\nabla\cdot(u_h^k\nabla\tilde{\mu}_u^k)\|_{L_h^2}+\dfrac{1}{h}\|\mathcal{R}^k\|_{H^{-1}}\right).
		\end{equation}
		Now we estimate the first term on the right hand side of \eqref{eq:error2}. Using the Young inequality, we have
		\begin{equation}\label{eq:r1}
		\begin{aligned}
		\tau\|\nabla\cdot(\tilde{u}^k\nabla\mathcal{V}_u^k)\|_{L_h^2}&\le\tau(\|\tilde{u}^k\Delta\mathcal{V}_u^k\|_{L_h^2}
		+\|\nabla\tilde{u}^k\nabla\mathcal{V}_u^k\|_{L_h^2})\\
		&\le\tau(\|\Delta\mathcal{V}_u^k\|_{\infty}\|\tilde{u}^k\|_{L_h^2}+\|\nabla\mathcal{V}_u^k\|_{\infty}\|\nabla\tilde{u}^k\|_{L_h^2})\\
		&\le C\tau(\|\tilde{u}^k\|_{L_h^2}+\|\nabla\tilde{u}^k\|_{L_h^2})\\
		&\le Ch^{15/4},
		\end{aligned}
		\end{equation}
		where \eqref{eq:tilde u} and $\tau\le Ch$ have been used in the last inequality. For the second term on the right hand side of \eqref{eq:error2}, we have
		\begin{equation}\label{eq:r2}
		\begin{aligned}
		\tau\|\nabla\cdot(u_h^k\nabla\tilde{\mu}_u^k)\|_{L_h^2}&\le \tau(\|\nabla u_h^k\nabla\tilde{\mu}_u^k\|_{L_h^2}+\|u_h^k\Delta\tilde{\mu}_u^k\|_{L_h^2})\\
		&\le\tau( \|\nabla u_h^k\|_{\infty}\|\nabla\tilde{\mu}_u^k\|_{L_h^2}+\|u_h^k\|_{\infty}\|\Delta\tilde{\mu}_u^k\|_{L_h^2})\\
		&\le C^*\tau(\|\nabla\tilde{\mu}_u^k\|_{L_h^2}+\|\Delta\tilde{\mu}_u^k\|_{L_h^2})\\
		&\le Ch^{11/4},
		\end{aligned}
		\end{equation}
		where \eqref{eq:mu} and the inverse inequality have been used in the last inequality. Substitution of \eqref{eq:r1}-\eqref{eq:r2} and \eqref{eq:r} into \eqref{eq:error2} leads to
		\begin{equation*}
		\|\tilde{u}^{k+1}-\tilde{u}^{k}\|_{L_h^2}\le Ch^{11/4}.
		\end{equation*}
		Then, we can obtain a rough estimate for $\tilde{u}^{k+1}$:
		\begin{equation*}
		\|\tilde{u}^{k+1}\|_{L_h^2}\le\|\tilde{u}^{k+1}-\tilde{u}^{k}\|_{L_h^2}+\|\tilde{u}^{k}\|_{L_h^2}\le Ch^{11/4}.
		\end{equation*}
		An application of the inverse inequality  implies that
		\begin{align*}
		\|\tilde{u}^{k+1}\|_{\infty}\le\dfrac{C\|\tilde{u}^{k+1}\|_{L_h^2}}{h}\le Ch^{7/4}\le 1,\\
		\|\nabla\tilde{u}^{k+1}\|_{\infty}\le\dfrac{C\|\tilde{u}^{k+1}\|_{\infty}}{h}\le Ch^{3/4}\le 1.
		\end{align*}
		We take $h$ sufficiently small such that
		\begin{equation*}
		\|\tilde{u}^{k+1}\|_{\infty}\le Ch^{7/4}\le \dfrac{\epsilon_0}{4}.
		\end{equation*}
		A combination of above with \eqref{eq:u} leads to the strict separation property:
		\begin{equation*}
		u_h^{k+1}\ge I_h\hat{u}^{k+1}-\|\tilde{u}^{k+1}\|_{\infty}\ge\dfrac{\epsilon_0}{4}.
		\end{equation*}
		In addition, we can obtain the following $W^{1,\infty}$ bound for the numerical solution at time step $t^{k+1}$:
		\begin{align*}
		\|u_h^{k+1}\|_{\infty}\le\|\hat{u}^{k+1}\|_{\infty}+\|\tilde{u}^{k+1}\|_{\infty}\le C+1\le C^*,\\
		\|\nabla u_h^{k+1}\|_{\infty}\le\|\nabla\hat{u}^{k+1}\|_{\infty}+\|\nabla\tilde{u}^{k+1}\|_{\infty}\le C+1\le C^*,
		\end{align*}
		which completes the proof.
	\end{proof}
	
	\subsection{Recovery of the assumption \eqref{eq:tilde u}}\label{recovery} In this subsection, the assumption will be recovered at the next time step.
	
	Taking $\varphi=\tilde{u}^{k+1}$ in \eqref{eq:rough}, we arrive at
	\begin{align*}
		\dfrac{1}{2\tau}\left(\|\tilde{u}^{k+1}\|_{L_h^2}^2-\|\tilde{u}^k\|_{L_h^2}^2\right)+(u_h^k\nabla\tilde{\mu}_u^k,\nabla\tilde{u}^{k+1})_h\le-(\tilde{u}^k\nabla\mathcal{V}_u^k,\nabla\tilde{u}^{k+1})_h+\left\langle R^k,\tilde{u}^{k+1}\right\rangle.
	\end{align*}
	The second term on the left hand side of above inequality can be rewritten as
		\begin{align*}
	&(u_h^k\nabla\tilde{\mu}_u^k,\nabla\tilde{u}^{k+1})_h=(u_h^k(\nabla(I_h\log\hat{u}^{k+1}-I_h\log u_h^{k+1})-\chi \nabla A_h\tilde{u}^k),\nabla\tilde{u}^{k+1})_h\\
	=&\Big(u_h^k\nabla I_h\Big(\dfrac{\hat{u}^{k+1}-u_h^{k+1}}{\xi}\Big),\nabla\tilde{u}^{k+1}\Big)_h-(\chi u_h^k\nabla  A_h\tilde{u}^k,\nabla\tilde{u}^{k+1})_h\\
	=& \Big(\dfrac{u_h^k}{I_h\xi}\nabla\tilde{u}^{k+1},\nabla\tilde{u}^{k+1}\Big)_h-\Big(\dfrac{u_h^k\nabla I_h\xi}{(I_h\xi)^2}\tilde{u}^{k+1},\nabla\tilde{u}^{k+1}\Big)_h-(\chi u_h^k\nabla  A_h\tilde{u}^k,\nabla\tilde{u}^{k+1})_h,
	\end{align*} 
	where $\xi$ lies between $\hat{u}^{k+1}$ and $u_h^{k+1}$. 
	 Utilizing the strict separation property and the $W^{1,\infty}$ bound for the numerical solution, we obtain the following estimate
	\begin{align*}
		&(u_h^k\nabla\tilde{\mu}_u^k,\nabla\tilde{u}^{k+1})_h\\
		\ge& \dfrac{\epsilon_0}{4C^*}\|\nabla\tilde{u}^{k+1}\|_2^2-\dfrac{16(C^*)^2}{\epsilon_0^2}|(\tilde{u}^{k+1},\nabla\tilde{u}^{k+1})_h|-(\chi  u_h^k \nabla A_h\tilde{u}^k,\nabla\tilde{u}^{k+1})_h
		\\
		\ge& \dfrac{\epsilon_0}{4C^*}\|\nabla\tilde{u}^{k+1}\|_{L_h^2}^2-\dfrac{\epsilon_0}{8C^*}\|\nabla \tilde{u}^{k+1}\|_{L_h^2}^2-C\|\tilde{u}^{k+1}\|_{L_h^2}^2-(\chi  u_h^k \nabla A_h\tilde{u}^k,\nabla\tilde{u}^{k+1})_h
		\\
		=& \dfrac{\epsilon_0}{8C^*}\|\nabla\tilde{u}^{k+1}\|_{L_h^2}^2-C\|\tilde{u}^{k+1}\|_{L_h^2}^2-(\chi  u_h^k \nabla A_h\tilde{u}^k,\nabla\tilde{u}^{k+1})_h,
	\end{align*}
	 where $C$ is a positive constant satisfying $\sqrt{C\epsilon_0/8C^*}\ge 8(C^*/\epsilon_0)^2$. 
	Combining above inequalities leads to
\begin{equation}\label{eq:refine,right}
		\begin{aligned}
	&\dfrac{1}{2\tau}\left(\|\tilde{u}^{k+1}\|_{L_h^2}^2-\|\tilde{u}^k\|_{L_h^2}^2\right)+\dfrac{\epsilon_0}{8C^*}\|\nabla\tilde{u}^{k+1}\|_{L_h^2}^2\\
	\le&\ C\|\tilde{u}^{k+1}\|_{L_h^2}^2+(\chi  u_h^k\nabla A_h\tilde{u}^k,\nabla\tilde{u}^{k+1})_h-(\tilde{u}^k\nabla\mathcal{V}_u^k,\nabla\tilde{u}^{k+1})_h+\left\langle R^k,\tilde{u}^{k+1}\right\rangle.
	\end{aligned}
\end{equation}
	Using Cauchy-Schwarz inequality and the strict separation property of the numerical solution, we have
	\begin{equation*}
	\begin{aligned}
	&	|(\chi  u_h^k \nabla A_h\tilde{u}^k,\nabla\tilde{u}^{k+1})_h|
		\le\chi\| u_h^k\|_{\infty}\|\nabla A_h\tilde{u}^k\|_{L_h^2}\|\nabla\tilde{u}^{k+1}\|_{L_h^2}
		\\\le\ & C\|\tilde{u}^k\|_{L_h^2}\|\nabla\tilde{u}^{k+1}\|_{L_h^2}
		\le \dfrac{\epsilon_0}{32C^*}\|\nabla\tilde{u}^{k+1}\|_{L_h^2}^2+C\|\tilde{u}^k\|_{L_h^2}^2.
	\end{aligned}
	\end{equation*}
	An application of \eqref{eq:v} shows that
	\begin{align*}
		|-(\tilde{u}^k\nabla\mathcal{V}_u^k,\nabla\tilde{u}^{k+1})_h|\le C\|\tilde{u}^k\|_{L_h^2}\|\nabla\tilde{u}^{k+1}\|_{L_h^2}\le \dfrac{\epsilon_0}{32C^*}\|\nabla\tilde{u}^{k+1}\|_{L_h^2}^2+C\|\tilde{u}^k\|_{L_h^2}^2.
	\end{align*}
	The last term on the right hand side of \eqref{eq:refine,right} can be estimated as follows
	\begin{align*}
		\Big|\left\langle R^k,\tilde{u}^{k+1}\right\rangle\Big|\le Ch^4\|\nabla\tilde{u}^{k+1}\|_2\le Ch^8+\dfrac{\epsilon_0}{32C^*}\|\nabla\tilde{u}^{k+1}\|_{L_h^2}^2.
	\end{align*}
	Substitution of above into \eqref{eq:refine,right} leads to
	\begin{align*}
		\dfrac{1}{2\tau}\left(\|\tilde{u}^{k+1}\|_{L_h^2}^2-\|\tilde{u}^k\|_{L_h^2}^2\right)+\dfrac{\epsilon_0}{32C^*}\|\nabla\tilde{u}^{k+1}\|_{L_h^2}^2\le Ch^8+C(\|\tilde{u}^k\|_{L_h^2}^2+\|\tilde{u}^{k+1}\|_{L_h^2}^2),
	\end{align*}
	multiplying $2\tau$ on both sides and summing this inequality from 0 to $k$ leads to
	\begin{align*}
		\|\tilde{u}^{k+1}\|_{L_h^2}^2+\sum_{i=0}^{k}\dfrac{\epsilon_0\tau}{16C^*}\|\nabla\tilde{u}^{i+1}\|_{L_h^2}^2\le Ch^8+\sum_{i=0}^{k+1}4C\tau\|\tilde{u}^i\|_{L_h^2}^2.
	\end{align*}
	Choosing $\tau$ sufficiently small such that $1-4C\tau>1/2$, using the Gronwall inequality (Lemma \ref{lem:gronwall}), we derive that
	\begin{align*}
		\|\tilde{u}^{k+1}\|_{L_h^2}^2+\sum_{i=0}^{k}\dfrac{\epsilon_0\tau}{8C^*}\|\nabla\tilde{u}^{i+1}\|_{L_h^2}^2\le C\exp\Big(\sum_{i=0}^{k}8C\tau\Big)h^8,
	\end{align*}
	then we obtain $\|\tilde{u}^{k+1}\|_{L_h^2}\le Ch^4\le Ch^{15/4}$, the assumption \eqref{eq:tilde u} is recovered at the next time step.
	
	\subsection{Proof of Theorem \ref{th.err}} In this subsection,
	we shall make use of  the strict separation property and the uniform $W^{1,\infty}$ bound for the numerical solution derived in the above  to prove Theorem \ref{th.err}.

	A weak formulation of \eqref{1.1}-\eqref{1} is
	\begin{align}
	(u_t,\varphi)+(\nabla u,\nabla\varphi) - \chi(u\nabla v,\nabla\varphi)=0, \ \forall\varphi\in H_0^1(\Omega), \label{eq:weak01}\\
	(\nabla v,\nabla\psi)+(v,\psi)=\alpha( u,\psi),\ \forall\psi\in H_0^1(\Omega).\label{eq:weak02}
	\end{align}
	Substituting $I_hu(t)$ into \eqref{eq:weak01} at $t=t_{k+1}$, we have
	\begin{equation}\label{eq:weak11}
	\begin{aligned}
	&\left(\overline{\partial}I_hu^k,\ \varphi\right)_h+\left(\nabla I_hu^{k+1},\ \nabla\varphi\right)_h-\chi\left(I_hu^k\nabla v^k,\ \nabla\varphi\right)_h\\
	=&\left(\overline{\partial}u^k-u_t^{k+1},\ \varphi\right)+\left(\overline{\partial}I_hu^k-\overline{\partial}u^k,\ \varphi\right)+\left(\nabla(I_hu^{k+1}-u^{k+1}),\ \varphi\right)\\
	&+\chi\left((u^k-I_hu^k)\nabla v^k,\ \varphi\right)+\chi\left(u^{k+1}\nabla v^{k+1}-u^k\nabla v^k,\ \varphi\right)\\
	&+\left(\overline{\partial}I_h u^k,\varphi\right)_h-\left(\overline{\partial}I_h u^k,\varphi\right)
	+\left(\nabla I_hu^{k+1},\nabla \varphi\right)_h-\left(\nabla  I_hu^{k+1},\nabla \varphi\right) \\
	&+\chi\left(I_hu^k\nabla v^k,\nabla\varphi\right)-\chi\left(I_hu^k\nabla v^k,\nabla\varphi\right)_h\\
	:=&I_1+I_2+I_3+I_4+I_5+I_6+I_7+I_8:=R_1,
	\end{aligned}
	\end{equation}
	where $R_1$ represents the truncation error. Similarly, substituting $I_hu(t),I_hv(t)$ into \eqref{eq:weak02} at $t=t_k$ leads to
	\begin{align}\label{eq:weak22}
	(\nabla I_h v^k,\ \nabla\psi)+(\I_h v^k,\ \psi)=&\alpha(I_h u^k,\ \psi)+(\nabla(I_hv(t_k)-v^k),\ \psi)\\
	&+(I_h v^k-v^k,\ \psi)+\alpha(u^k-I_h u^k,\ \psi).\nonumber
	\end{align}
	We rewrite the numerical scheme \eqref{eq:fem1} as
	\begin{equation}\label{eq:fem11}
	\begin{aligned}
	&\left(\overline{\partial} u_h^k,\varphi\right)_h +\left(\nabla u_h^{k+1},\nabla \varphi\right)_h-\chi\left(u_h^k\nabla v_h^k,\nabla \varphi\right)_h\\
	=&\left(u_h^k\nabla(I-I_h)\log u_h^{k+1},\nabla\varphi\right)_h+\left((u_h^{k+1}-u_h^k)\nabla\log u_h^{k+1},\nabla\varphi\right)_h\\
	:=&I_9+I_{10}:=R_2.
	\end{aligned}
	\end{equation}

	We split the error functions as
	\begin{equation*}
	\begin{aligned}
	e_u^k&=u^k-u_h^k= (u^k-I_hu^k)+(I_hu^k-u_h^k):=\rho_u^k+\sigma_u^k,\\
	e_v^k&=v^k-v_h^k=(v^k-I_hv^k)+(I_hv^k-v_h^k):=\rho_v^k+\sigma_v^k.
	\end{aligned}
	\end{equation*}
	Then using the property of the interpolation \eqref{eq:I_h}, we have
	\begin{equation}\label{interpolation}
	\begin{aligned}
	\|\rho_u^k\|_{L^2(\Omega)}+h\|\nabla \rho_u^k\|_{L^2(\Omega)}\le Ch^2\|u\|_{H^2(\Omega)},\\
	\|\rho_v^k\|_{L^2(\Omega)}+h\|\nabla \rho_v^k\|_{L^2(\Omega)}\le Ch^2\|v\|_{H^2(\Omega)}.
	\end{aligned}
	\end{equation}
	Subtracting the numerical scheme formulation \eqref{eq:fem11} and \eqref{eq:fem2} from the weak form \eqref{eq:weak11} and \eqref{eq:weak22}, we obtain the following error equations:
	\begin{align}
	\left(\overline{\partial}\sigma_u^k,\ \varphi\right)_h+\left(\nabla\sigma_u^{k+1},\ \nabla\varphi\right)&_h=I_{11}+R_1-R_2,\ \label{eq:error31}\\
	(\nabla\sigma_v^{k},\ \nabla\psi)+(\sigma_v^k,\ \psi)=\alpha&(\sigma_u^k,\ \psi)-\alpha\epsilon_h(u_h^{k},\ \psi)+(\nabla(I_hv-v),\ \psi)\ \label{eq:error32}\\
	&+(I_h v-v,\ \psi)+\alpha(u-I_h u,\ \psi),\nonumber
	\end{align}
	for all $\varphi,\psi\in X_h,k\ge 1$, where $R_1,R_2$ are defined before and $I_{11}$ is defined as follows
	\begin{align*}
	I_{11}=\chi\left(I_hu(t_k)\nabla v^k-u_h^k\nabla v_h^k,\ \nabla\varphi\right)_h.
	\end{align*}
	Taking $\varphi=\sigma_u^{k+1}$ in \eqref{eq:error31} leads to
	\begin{equation}\label{eq:err}
	\begin{aligned}
	\dfrac{1}{2\tau}\left(\|\sigma_u^{k+1}\|_{L_h^2}^2-\|\sigma_u^{k}\|_{L_h^2}^2+\|\sigma_u^{k+1}-\sigma_u^k\|_{L_h^2}^2\right)+\|\nabla \sigma_u^{k+1}\|_{L_h^2}^2=I_{11}+R_1-R_2.
	\end{aligned}
	\end{equation}

	Now we estimate the terms on the right-hand side of \eqref{eq:I11-0}. For the first term $I_{11}$, applying the Cauchy-Schwarz inequality and the Young inequality yields
	\begin{equation}\label{eq:I11-0}
	\begin{aligned}
	&|I_{11}|=|\chi\left(\sigma_u^k\nabla v^k+u_h^k\nabla e_v^k,\ \nabla\sigma_u^{k+1}\right)_h|\\
	\le&\ \|\sigma_u^k\nabla v^k\|_{L_h^2(\Omega)}\|\nabla\sigma_u^{k+1}\|_{L_h^2(\Omega)}+\|u_h^k\nabla e_v^k\|_{L_h^2(\Omega)}\|\nabla\sigma_u^{k+1}\|_{L_h^2(\Omega)}\\
	\le&\ C\|\sigma_u^k\|_{L_h^2(\Omega)}^2+\dfrac{1}{20}\|\nabla\sigma_u^{k+1}\|_{L_h^2(\Omega)}^2+C\|\nabla e_v^k\|_{L_h^2(\Omega)}^2.
	\end{aligned}
	\end{equation}
	In order to estimate $\|\nabla e_v^k\|_{L_h^2(\Omega)}^2$ above, taking $\psi=\sigma_v^k$ in \eqref{eq:error32} and applying Lemma \ref{lem:qua} leads to
	\begin{equation*}
	\begin{aligned}
	\|\nabla\sigma_v^k\|_{L^2(\Omega)}^2+\|\sigma_v^k\|_{L^2(\Omega)}^2\le&\  \alpha\|\sigma_u^k\|_{L^2(\Omega)}\|\sigma_v^k\|_{L^2(\Omega)}+Ch^2\|\nabla u_h^k\|_{L^2(\Omega)}\|\nabla\sigma_v^k\|_{L^2(\Omega)}\\
	&+Ch\|v\|_{H^2(\Omega)}\|\sigma_v^k\|_{L^2(\Omega)}+Ch^2\|u\|_{H^2(\Omega)}\|\sigma_v^k\|_{L^2(\Omega)}\\
	\le&\ C\|\sigma_u^k\|_{L^2(\Omega)}^2+\|\sigma_v^k\|_{L^2(\Omega)}^2+\dfrac{1}{2}\|\nabla\sigma_v^k\|_{L^2(\Omega)}^2+Ch^2,
	\end{aligned}
	\end{equation*}
	where the property of the interpolation has been used in the first inequality. Thus we obtain the following estimate for $\|\nabla \sigma_v^k\|_{L^2(\Omega)}$:
	\begin{align}\label{errorv}
	\|\nabla\sigma_v^k\|_{L^2(\Omega)}^2\le C\|\sigma_u^k\|_{L^2(\Omega)}^2+Ch^2.
	\end{align}
	Applying Lemma  \ref{lem:qua} indicates that
	\begin{align*}
		\|\nabla\sigma_v^k\|_{L_h^2(\Omega)}^2\le C\|\sigma_u^k\|_{L_h^2(\Omega)}^2+Ch^2\|\nabla\sigma_u^k\|_{L^2(\Omega)}^2+Ch^2.
	\end{align*}
	A combination of the above estimates for $\|\nabla\sigma_v^k\|_{L_h^2(\Omega)}$ with \eqref{interpolation} leads to
	\begin{align}\label{errorv2}
		\|\nabla e_v^k\|_{L_h^2(\Omega)}^2\le C\|\sigma_u^k\|_{L_h^2(\Omega)}^2+Ch^2\|\nabla\sigma_u^k\|_{L^2(\Omega)}^2+Ch^2.
	\end{align}
	Substitution of above into \eqref{eq:I11-0} leads to
	\begin{equation}\label{eq:I11}
	\begin{aligned}
	|I_{11}|&\le C\|\sigma_u^k\|_{L_h^2(\Omega)}^2+\dfrac{1}{20}\|\nabla\sigma_u^{k+1}\|_{L_h^2(\Omega)}^2+Ch^2\|\nabla\sigma_u^k\|_{L^2(\Omega)}^2+Ch^2.
	\end{aligned}
	\end{equation}
	
	Next we estimate the second term $R_1=I_1+I_2+I_3+I_4+I_5+I_6+I_7+I_8$.
	For the first term,  we can derive  \cite{2006Thomee}
	\begin{equation}\label{eq:I1}
	\begin{aligned}
	|I_1|&=|\left(\overline{\partial}u^k-u_t^{k+1},\ \sigma_u^{k+1}\right)|
	\le C\tau^{1/2}\Big(\int_{t_k}^{t_{k+1}}\|u_{tt}\|_{L^2(\Omega)}^2\mathrm{d}s\Big)^{1/2}\|\sigma_u^{k+1}\|_{L^2(\Omega)}\\
	&\le C\tau\int_{t_k}^{t_{k+1}}\|u_{tt}\|_{L^2(\Omega)}^2\mathrm{d}s+\dfrac{1}{20}\|\nabla\sigma_u^{k+1}\|_{L^2(\Omega)}^2.
	\end{aligned}
	\end{equation}
	An application of the property of the interpolation and the Young inequality leads to
	\begin{equation}\label{eq:I2}
	\begin{aligned}
	|I_2|=&|\left(\overline{\partial}I_hu^k-\overline{\partial}u^k,\ \sigma_u^{k+1}\right)|\\
	\le&\ \dfrac{1}{\tau}\|I_h(u^{k+1}-u^k)-(u^{k+1}-u^k)\|_{L^2(\Omega)}\|\sigma_u^{k+1}\|_{L^2(\Omega)} \\
	=&\ \dfrac{1}{\tau}\|(I_hu_t-u_t)\tau\|_{L^2(\Omega)}\|\sigma_u^{k+1}\|_{L^2(\Omega)}\\
	\le&\ Ch^2\|u_t\|_{H^2(\Omega)}\|\nabla\sigma_u^{k+1}\|_{L^2(\Omega)}\\
	\le&\ Ch^4+\dfrac{1}{20}\|\nabla\sigma_u^{k+1}\|_{L^2(\Omega)}^2.
	\end{aligned}
	\end{equation}
	Using the Cauchy-Schwarz inequality and the property of the interpolation, we have
	\begin{equation}\label{eq:I3}
	\begin{aligned}
	|I_3|=&|\left(\nabla(I_hu^{k+1}-u^{k+1}),\ \nabla\sigma_u^{k+1}\right)|\\
	\le&\|\nabla(I_hu^{k+1}-u^{k+1})\|_{L^2(\Omega)}\| \nabla\sigma_u^{k+1}\|_{L^2(\Omega)}\\
	\le&Ch\|u^{k+1}\|_{H^2(\Omega)}\| \nabla\sigma_u^{k+1}\|_{L^2(\Omega)}\\
	\le&Ch^2+\dfrac{1}{20}\|\nabla\sigma_u^{k+1}\|_{L^2(\Omega)}^2.
	\end{aligned}
	\end{equation}
	Similarly, we have
	\begin{equation}\label{eq:I4}
	\begin{aligned}
	|I_4|=&|\chi\left((u^k-I_hu^k)\nabla v^k,\ \nabla\sigma_u^{k+1}\right)|\\
	\le&\ \chi \|u^k-I_hu^k\|_{L^2(\Omega)}\|\nabla v^k\|_{L^{\infty}(\Omega)}\| \nabla\sigma_u^{k+1}\|_{L^2(\Omega)}\\
	\le&\ Ch^2\|u^k\|_{H^2(\Omega)}\| \nabla\sigma_u^{k+1}\|_{L^2(\Omega)}\\
	\le&\ Ch^4+\dfrac{1}{20}\| \nabla\sigma_u^{k+1}\|_{L^2(\Omega)}^2.
	\end{aligned}
	\end{equation}
	Now we apply the Cauchy-Schwarz inequality and the Young inequality:
	\begin{equation}\label{eq:I5}
	\begin{aligned}
	|I_5|=&|\chi\left(\tau u_t\nabla v^{k+1}+\tau u^k\nabla v_t,\ \nabla\sigma_u^{k+1}\right)|\\
	\le& \ C\left(\|u_t\|_{L^{\infty}(\Omega)}\|\nabla v\|_{L^{\infty}(\Omega)}+
	\|u\|_{L^{\infty}(\Omega)}\|\nabla v_t\|_{L^{\infty}(\Omega)}\right)\tau\|\nabla \sigma_u^{k+1}\|_{L^2(\Omega)}\\
	\le&\ C\tau^2+\dfrac{1}{20}\|\nabla \sigma_u^{k+1}\|_{L^2(\Omega)}^2.
	\end{aligned}
	\end{equation}
	Notice that
		\begin{align*}
	I_6&=\dfrac{1}{\tau}\left(I_hu^{k+1}-I_hu^k,\ \sigma_u^{k+1}\right)-\dfrac{1}{\tau}\left(I_hu^{k+1}-I_hu^k,\ \sigma_u^{k+1}\right)_h\\
	&=-\epsilon_h\left(I_hu_t,\ \sigma_u^{k+1}\right),
	\end{align*}
	therefore, using  Lemma \ref{lem:qua} leads to the following estimate for $I_6$:
	\begin{align}\label{eq:I6}
	|I_6|=|\epsilon_h(I_hu_t,\sigma_u^{k+1})|\le Ch^2\|\nabla I_hu_t\|\|\nabla\sigma_u^{k+1}\|\le Ch^4+\dfrac{1}{20}\|\nabla\sigma_u^{k+1}\|_{L^2(\Omega)}^2.
	\end{align}
	We recall the quadrature formula defined before and Lemma \ref{lem:qua}, and arrive at
	\begin{align}\label{eq:I7}
	I_7&=\left(I_hu^{k+1},\nabla \sigma_u^{k+1}\right)_h-\left(I_hu^{k+1},\nabla \sigma_u^{k+1}\right)=0,
	\end{align}
	An application of the Cauchy-Schwarz inequality and the property of the interpolation leads to
	\begin{equation}\label{eq:I8}
		\begin{aligned}
		|I_8|&=|\chi\left(I_hu^k\nabla v^k,\nabla\sigma_u^{k+1}\right)-\chi\left(I_hu^k\nabla  v^k,\nabla\sigma_u^{k+1}\right)_h|\\
		&=|\chi\left(I_hu^k\nabla(I-I_h) v^k,\nabla\sigma_u^{k+1}\right)-\chi\left(I_hu^k\nabla (I-I_h) v^k,\nabla\sigma_u^{k+1}\right)_h|\\
		&\le Ch\|I_hu^k\|_{L^{\infty}}\|u^k\|_{L^2(\Omega)}\|\nabla\sigma_u^{k+1}\|_{L^2(\Omega)}\\
		&\le Ch^2+\dfrac{1}{20}\|\nabla\sigma_u^{k+1}\|_{L^2(\Omega)}^2.
		\end{aligned}
	\end{equation}
	Substituting estimates \eqref{eq:I1}-\eqref{eq:I8} into $R_1$ and applying the property $\eqref{norm}$, we obtain
	\begin{equation}\label{eq:R1}
	\begin{aligned}
	|R_1|\le& C\tau\int_{t_k}^{t_{k+1}}\|u_{tt}\|_{L^2(\Omega)}^2\mathrm{d}s+\dfrac{7}{20}\|\nabla\sigma_u^{k+1}\|_{L^2(\Omega)}^2+Ch^2+C\tau^2\\
	=&C\tau\int_{t_k}^{t_{k+1}}\|u_{tt}\|_{L^2(\Omega)}^2\mathrm{d}s+\dfrac{7}{20}\|\nabla\sigma_u^{k+1}\|_{L_h^2(\Omega)}^2+Ch^2+C\tau^2.
	\end{aligned}
	\end{equation}

	It remains to bound each term of $R_2=I_9+I_{10}$. Now we use the Cauchy-Schwarz inequality and the Young inequality:
	\begin{equation}\label{eq:I9}
	\begin{aligned}
	|I_9|=&|\left(u_h^k\nabla(I-I_h)\log u_h^{k+1},\nabla\sigma_u^{k+1}\right)_h|\\
	\le&\ \|u_h^k\nabla(I-I_h)\log u_h^{k+1}\|_{L_h^2(\Omega)}\|\nabla\sigma_u^{k+1}\|_{L_h^2(\Omega)}\\
	\le&\ \|u_h^k\|_{L^{\infty}(\Omega)}\|\nabla(I-I_h)\log u_h^{k+1}\|_{L_h^2(\Omega)}\|\nabla\sigma_u^{k+1}\|_{L_h^2(\Omega)}\\
	\le&\ Ch^2+\dfrac{1}{20}\|\nabla\sigma_u^{k+1}\|_{L_h^2(\Omega)},
	\end{aligned}
	\end{equation}
	where we have used the following inequality:
	\begin{align*}
	\|\nabla(I-I_h)\log u_h^{k+1}\|_{L_h^2(\Omega)}
	\le&\ Ch\sum_{e}\sum_{|\alpha|=2}\Big(\int_{e}|D^{\alpha}\log u_h^{k+1}|^2\mathrm{d}x\Big)^{1/2}\\
	\le&\ Ch\sum_{e}\Big(\int_{e}\dfrac{1}{(u_h^{k+1})^4}|\nabla u_h^{k+1}|^4\mathrm{d}x\Big)^{1/2}\\
	\le&\ Ch\dfrac{1}{(\epsilon_0/4)^4}\sum_{e}\Big(\int_{e}|\nabla u_h^{k+1}|^4\mathrm{d}x\Big)^{1/2}\\
	\le &Ch.
	\end{align*}
	An application of the strict separation property and the $W^{1,\infty}$ bound of the numerical solution leads to
	\begin{equation}\label{eq:I10}
	\begin{aligned}
	|I_{10}|=&|\left((u_h^{k+1}-u_h^k)\nabla\log u_h^{k+1},\nabla\sigma_u^{k+1}\right)_h|\\
	\le&\ \|\dfrac{u_h^{k+1}-u_h^k}{u_h^{k+1}}\nabla u_h^{k+1}\|_{L_h^2(\Omega)}\|\nabla\sigma_u^{k+1}\|_{L_h^2(\Omega)}\\
	\le&\ \dfrac{\|\nabla u_h^{k+1}\|_{L^{\infty}(\Omega)}}{\epsilon_0/4}\left(\|\sigma_u^{k+1}-\sigma_u^k\|_{L_h^2(\Omega)}+C\tau\right)\|\nabla\sigma_u^{k+1}\|_{L_h^2(\Omega)}\\
	\le&\ C(\|\sigma_u^{k+1}-\sigma_u^k\|_{L_h^2(\Omega)}+C\tau)\|\nabla\sigma_u^{k+1}\|_{L_h^2(\Omega)}\\
	\le&\  C\|\sigma_u^{k+1}-\sigma_u^k\|_{L_h^2(\Omega)}^2+C\tau^2+\dfrac{1}{20}\|\nabla\sigma_u^{k+1}\|_{L_h^2(\Omega)}^2.
	\end{aligned}
	\end{equation}
	Combining the estimates \eqref{eq:I9}-\eqref{eq:I10}, we obtain
	\begin{equation}\label{eq:R2}
	\begin{aligned}
	|R_2|\le C\|\sigma_u^{k+1}-\sigma_u^k\|_{L_h^2(\Omega)}^2+\dfrac{1}{10}\|\nabla \sigma_u^{k+1}\|_{L_h^2(\Omega)}^2+Ch^2+C\tau^2.
	\end{aligned}
	\end{equation}

	Finally, combining \eqref{eq:I11},\eqref{eq:R1} and \eqref{eq:R2} in \eqref{eq:err}, we find
	\begin{equation}\label{eq:err1}
	\begin{aligned}
	\dfrac{1}{2\tau}&\left(\|\sigma_u^{k+1}\|_{L_h^2(\Omega)}^2-\|\sigma_u^{k}\|_{L_h^2(\Omega)}^2+\|\sigma_u^{k+1}-\sigma_u^k\|_{L_h^2(\Omega)}^2\right)+\|\nabla \sigma_u^{k+1}\|_{L_h^2(\Omega)}^2\\
	\le&\ C\|\sigma_u^k\|_{L^2(\Omega)}^2+C\|\sigma_u^{k+1}-\sigma_u^k\|_{L_h^2(\Omega)}^2+\dfrac{1}{2}\|\nabla \sigma_u^{k+1}\|_{L_h^2(\Omega)}^2+Ch^2+C\tau^2\\
	&+Ch^2\|\nabla\sigma_u^k\|_{L^2(\Omega)}^2+C\tau\int_{t_k}^{t_{k+1}}\|u_{tt}\|_{L^2(\Omega)}^2\mathrm{d}s.
	\end{aligned}
	\end{equation}
	Multiplying by $2\tau$ on both sides of \eqref{eq:err1} and summing up from $k=0$ to $m-1$, we get
	\begin{equation*}
	\begin{aligned}
	\|\sigma_u^{m}\|_{L_h^2(\Omega)}^2-\|\sigma_u^{0}\|_{L_h^2(\Omega)}^2+(1-2C\tau)\|\sigma_u^{k+1}-\sigma_u^k\|_{L_h^2(\Omega)}^2+\tau\sum_{k=0}^{m-1}\|\nabla \sigma_u^{k+1}\|_{L_h^2(\Omega)}^2\\
	\le C\tau\sum_{k=0}^{m-1}\|\sigma_u^k\|_{L_h^2(\Omega)}^2+Ch^2+C\tau^2+C\tau h^2\sum_{k=0}^{m-1}\|\nabla\sigma_u^k\|_{L^2(\Omega)}^2+C\tau^2\int_{0}^{T}\|u_{tt}\|_{L^2}^2\mathrm{d}s.
	\end{aligned}
	\end{equation*}
	
	Assuming $1-2C\tau>0$, $1-Ch^2>1/2$  since $\tau$ and $h$ is small enough, and
	applying the discrete Gronwall inequality (Lemma \ref{lem:gronwall})  to the above leads to
	\begin{equation*}
	\|\sigma_u^{m}\|_{L_h^2(\Omega)}+\Big(\tau\sum_{k=0}^{m-1}\|\nabla \sigma_u^{k+1}\|_{L_h^2(\Omega)}^2\Big)^{1/2}\le C(h+\tau),\ \forall m\in\mathbb{N}.
	\end{equation*}
	A combination of the above estimates for $\|\sigma_u^{m}\|_{L_h^2(\Omega)}$ and $\|\nabla \sigma_u^{k+1}\|_{L_h^2(\Omega)}$ with \eqref{interpolation} leads to the desired error estimates for $u$. Finally, we obtain the error estimates for $v$ from \eqref{errorv} and  \eqref{errorv2}.
	
	The proof of {\bf Theorem  \ref{th.err}} is complete.
	
	\section{Finite-time blowup}\label{sec4}
	In this section, we  discuss whether the solution of the fully discrete scheme \eqref{eq:fem1}-\eqref{eq:fem2} will blow up in finite time.

	We first prove a discrete analog of  Lemma \ref{lem:blowup}.
	Taking $\varphi=I_h\Phi$ in \eqref{eq:fem1}, where $\Phi$ is defined as in Lemma \ref{lem:blowup}, then from \eqref{eq:I_h}, we have the following error estimate
	\begin{equation}\label{eq:I_h2}
	\|\varphi-\Phi\|_{L^2(\Omega)}+h\|\varphi-\Phi\|_{L^2(\Omega)}\le Ch^2\|\Phi\|_{H^2(\Omega)}.
	\end{equation}
	\begin{lemma}\label{lem:mk}
		Assume that $u$ is smooth enough for a fix time $T>0$, let $M^k=(u_h^k,\ \varphi)_h,\ \theta=(u,1)$. Under the mild mesh-sizes requirement $\tau\le Ch$, if $(u_h^0,\varphi)_h<\infty$,
		then $(u_h^k,\varphi)_h<\infty$ and the following inequality holds
		\begin{equation*}
		\frac{M^{k+1}-M^k}{\tau}\le 4\theta-\dfrac{\alpha\chi}{2\pi}\theta^2+C_1\theta M^k+C_2\theta^{3/2}(M^k+Ch)^{1/2}+C_0h.
		\end{equation*}
	\end{lemma} 
	\begin{proof}
		We rewrite  \eqref{eq:fem1} as 
		\begin{equation}\label{eq:bl0}
		\begin{aligned}
		\frac{M^{k+1}-M^k}{\tau}&=-(\nabla u^{k+1},\nabla\Phi)+\chi(u^k\nabla v^k,\nabla \Phi)+\sum_{i=1}^{8}J_i\\
		&\le 4\theta-\dfrac{\alpha\chi}{2\pi}\theta^2+C_1\theta(u^k,\Phi)+C_2\theta^{3/2}(u^k,\Phi)^{1/2}+\sum_{i=1}^{8}J_i,
		\end{aligned}
		\end{equation}
		where Lemma \ref{lem:blowup} has been used in the last inequality, and $J_i\ (i=1,\cdots,8)$ are defined as follows
		\begin{equation*}
		\begin{aligned}
		J_1=&(\nabla u^{k+1},\ \nabla(\Phi-\varphi)),\\
		J_2=&(\nabla (u^{k+1}-u_h^{k+1}),\ \nabla \varphi),\\
		J_3=&\left(\nabla u_h^{k+1},\nabla\varphi\right)-\left(\nabla u_h^{k+1},\nabla\varphi\right)_h,\\
		J_4=&\left((u_h^{k+1}-u_h^k)\nabla \log u_h^{k+1},\nabla\varphi\right)_h,\\
		J_5=&\chi\left(u^k\nabla v^k,\nabla(\varphi-\Phi)\right),\\
		J_6=&\chi(u_h^k\nabla v_h^k-u^k\nabla v^k,\nabla\varphi),\\
		J_7=&\chi(u_h^k\nabla v_h^k,\nabla\varphi)_h-\chi(u_h^k\nabla v_h^k,\nabla\varphi),\\
		J_8=&\left(u_h^k\nabla((I-I_h)\log u_h^{k+1}),\nabla\varphi\right)_h.
		\end{aligned}
		\end{equation*}
		Thanks to \eqref{eq:I_h2}, we obtain
		\begin{equation*}
		\begin{aligned}
		&(u^k,\Phi)=(u^k,\Phi-\varphi)+(u^k-u_h^k,\varphi)+(u_h^k,\varphi)-(u_h^k,\varphi)_h+(u_h^k,\varphi)_h\\
		\le&\ \|u^k\|_{L^2(\Omega)}\|\Phi-\varphi\|_{L^2(\Omega)}+\|e_u^k\|_{L^2(\Omega)}\|\varphi\|_{L^2(\Omega)}
		+Ch^2\|\nabla u_h^k\|_{L^2(\Omega)}\|\nabla\varphi\|_{L^2(\Omega)}+M^k\\
		\le& \ Ch+M^k,
		\end{aligned}
		\end{equation*}
		where the error estimates and $\tau\le Ch$ have been used in the last inequality.
		Substitution of above into \eqref{eq:bl0} leads to
		\begin{equation}\label{eq:bl1}
		\frac{M^{k+1}-M^k}{\tau}\le 4\theta-\dfrac{\alpha\chi}{2\pi}\theta^2+C_1\theta M^k+C_2\theta^{3/2}(M^k+Ch)^{1/2}+\sum_{i=1}^{8}J_i.
		\end{equation}
		Next, we  estimate $J_i\ (i=1,\cdots,8)$ respectively. Utilizing the property of the interpolation operator in \eqref{eq:I_h2}, we have
		\begin{equation}\label{eq:J1}
		|J_1|\le \|\nabla u^{k+1}\|_{L^2(\Omega)}\|\nabla(\Phi-\varphi)\|_{L^2(\Omega)}\le Ch.
		\end{equation}
    We derive from  {\bf Theorem  \ref{th.err}}   that
    \begin{equation}\label{eq:J2}
	\begin{aligned}
	|J_2|&=|(\nabla (u^{k+1}-u_h^{k+1}),\ \nabla (\varphi-\Phi))+(\nabla (u^{k+1}-u_h^{k+1}),\ \nabla \Phi)|\\
	&=|(\nabla (u^{k+1}-u_h^{k+1}),\ \nabla (\varphi-\Phi))-( (u^{k+1}-u_h^{k+1}),\Delta \Phi)|\\
	&\le Ch\|\nabla e_u^k\|_{L^2(\Omega)}\|\Phi\|_{H^2(\Omega)}+\| e_u^k\|_{L^2(\Omega)}\|\Delta \Phi\|_{L^2(\Omega)}\\
	&\le Ch,
	\end{aligned}
\end{equation}
	where the property of interpolation and $\tau\le Ch$ have been used. 
	Noticing the definition of $(\cdot,\cdot)_h$, we have
		\begin{equation}\label{eq:J3}
		J_3=\left(\nabla u_h^{k+1},\nabla\varphi\right)-\left(\nabla u_h^{k+1},\nabla\varphi\right)_h=0.
		\end{equation}
		An application of the strict separation property, the $W^{1,\infty}$ bound for the numerical solution and $\tau\le Ch$ indicates that
		\begin{equation}\label{eq:J4}
		\begin{aligned}
		|J_4|&=\Big|\big(\dfrac{u_h^{k+1}-u_h^k}{u_h^{k+1}}\nabla u_h^{k+1},\nabla\varphi\Big)_h\Big|\\
		&\le\Big\|\dfrac{u_h^{k+1}-u_h^k}{u_h^{k+1}}\nabla u_h^{k+1}\Big\|_{L_h^2(\Omega)}\|\nabla\varphi\|_{L_h^2(\Omega)}\\
		&\le\dfrac{\|\nabla u_h^{k+1}\|_{L^{\infty}(\Omega)}}{\epsilon_0/4}\|u_h^{k+1}-u_h^k\|_{L_h^2(\Omega)}\|\nabla\varphi\|_{L_h^2(\Omega)}\\
		&\le \dfrac{C^*}{\epsilon_0/4}(\|e_u^{k+1}-e_u^k\|_{L_h^2(\Omega)}+C\tau)\|\nabla\varphi\|_{L_h^2(\Omega)}\\
		&\le C(\|e_u^{k+1}\|_{L_h^2(\Omega)}+\|e_u^k\|_{L_h^2(\Omega)}+C\tau)\|\nabla\varphi\|_{L_h^2(\Omega)}\\
		&\le Ch.
		\end{aligned}
		\end{equation}
		Utilizing the property of the interpolation operator in \eqref{eq:I_h2}, we have
		\begin{equation}\label{eq:J5}
		\begin{aligned}
		|J_5|&\le \chi\|u^k\nabla v^k\|_{L^2(\Omega)}\|\nabla(\varphi-\Phi)\|_{L^2(\Omega)}\\
		&\le Ch\|u^k\|_{L^{\infty}(\Omega)}\|\nabla v^k\|_{L^2(\Omega)}\|\Phi\|_{H^2(\Omega)}\\
		&\le Ch\|u^k\|_{L^{\infty}(\Omega)}\|u^k\|_{L^2(\Omega)}\|\Phi\|_{H^2(\Omega)}\\
		&\le Ch.
		\end{aligned}
		\end{equation}
	    An application of the Cauchy-Schwarz inequality and the error estimates leads to
		\begin{equation}\label{eq:J6}
		\begin{aligned}
		&|J_6|=|-\chi(u_h^k\nabla e_v^k+e_u^k\nabla v^k,\nabla\varphi)|\\
		\le& \chi(\|u_h^k\|_{L^{\infty}(\Omega)}\|\nabla e_v^k\|_{L^2(\Omega)}\|\nabla\varphi\|_{L^2(\Omega)}+\|e_u^k\|_{L^2(\Omega)}\|\nabla v^k\|_{L^{2}(\Omega)}\|\nabla\varphi\|_{L^{\infty}(\Omega)})\\
		\le& Ch,
		\end{aligned}
		\end{equation}
		where $\tau\le Ch$ has been used in the last inequality.
		An application of Lemma \ref{lem:qua} to $J_4$ leads to
		\begin{equation}\label{eq:J7}
		|J_7|=|\chi(u_h^k\nabla v_h^k,\nabla\varphi)_h-\chi(u_h^k\nabla v_h^k,\nabla\varphi)|=0.
		\end{equation}
		Using the property of the interpolation operator $I_h$ as in the proof of the Theorem \ref{th.err}, we have
		\begin{align}
		|J_8|=&|\left(u_h^k\nabla((I-I_h)\log u_h^{k+1}),\nabla\varphi\right)_h|\le Ch.\label{eq:J8}
		\end{align}
		Finally, a substitution of \eqref{eq:J1}-\eqref{eq:J8} into \eqref{eq:bl1} implies that
		\begin{equation*}
		\frac{M^{k+1}-M^k}{\tau}\le 4\theta-\dfrac{\alpha\chi}{2\pi}\theta^2+C_1\theta M^k+C_2\theta^{3/2}(M^k+Ch)^{1/2}+C_0h.
		\end{equation*}
		The proof is complete.
	\end{proof}
\begin{remark}
	The positive constant $C_0$ in Lemma \ref{lem:mk} depends on the regularity of the exact solutions.
\end{remark}
	We can then derive  the following discrete analog of  {\bf Theorem \ref{thm:blowup}}.
 	\begin{Theorem}\label{th:blow}
		Assume that $\theta=(u,1)>8\pi/(\alpha\chi)$. If $(u_h^0,\varphi)_h$ is sufficiently small, $h$ and $\tau\le Ch$ are sufficiently small,
		then the solution $(u_h^k,v_h^k)$ to the fully discrete scheme \eqref{eq:fem1}-\eqref{eq:fem2} will blow up in finite time, namely the maximal existence time $T_{\max}=k_{\max}\tau$ of the discrete solutions is finite.
	\end{Theorem} 
	\begin{proof}
		Obviously, we can derive the following inequality from Lemma \ref{lem:mk}
		\begin{align*}
				\frac{M^{k+1}-M^k}{\tau}\le 4\theta-\dfrac{\alpha\chi}{2\pi}\theta^2+C_1\theta M^k+C_2\theta^{3/2}(M^k)^{1/2}+C_3h^{1/2}.
		\end{align*}
		Denote $\beta=- (4\theta-\frac{\alpha\chi}{2\pi}\theta^2+C_3h^{1/2})$, we have the following inequality 
		\begin{align}\label{eq:induc}
		\frac{M^{k+1}-M^k}{\tau}\le -\beta+C_1\theta M^k+C_2\theta^{3/2}\sqrt{M^k}.
		\end{align}
		Under the condition that $\theta=(u,1)>8\pi/(\alpha\chi)$, we can choose sufficiently small $h$ such that $\beta>0$. Since $M^0=(u_h^0,\varphi)_h$ is sufficiently small, we have
		\begin{align*}
			\beta_0:=\beta-C_1\theta M^0-C_2\theta^{3/2}\sqrt{M^0}>0.
		\end{align*}
		
		We claim that the following inequality holds for all $k\in\mathbb{N}$
			\begin{align}\label{eq:mk}
				M^{k+1}\le M^k-\tau\beta_0.
			\end{align}
		We prove the above inequality by induction. Using the inequality \eqref{eq:induc} for $k=0$, we have
		\begin{align*}
		M^1\le M^0-\tau(\beta-C_1\theta M^0-C_2\theta^{3/2}\sqrt{M^0})=M^0-\tau\beta_0.
		\end{align*}
		Now assume that \eqref{eq:mk} holds for $k-1$, we have
		\begin{align*}
		M^{k}\le M^{k-1}-\tau\beta_0.
		\end{align*}
		Notice that $M^k$ is decreasing about $k$, we have
		\begin{align*}
		M^{k+1}&\le M^k-\tau(\beta-C_1\theta M^k-C_2\theta^{3/2}\sqrt{M^k})\\
		&\le M^k-\tau(\beta-C_1\theta M^{0}-C_2\theta^{3/2}\sqrt{M^{0}})\\
		&= M^k-\tau\beta_0.
		\end{align*}
		
		Next summing \eqref{eq:mk} over $k$ shows that
		\begin{align*}
			M^{k+1}\le M^0-
			(k+1)\tau\beta_0.
		\end{align*}
		 Hence, if the solution $(u_h^k,v_h^k)$ exists for all $k\ge0$, then $M^{k+1}$ becomes negative provided that $T>M^0/\beta_0$. This is a contradiction to the positivity of $M^k$. 
		 Thus, the proof is complete.
	\end{proof}
	\begin{remark}
		Note that in the classical Keller-Segel system, the solution may blow up in finite time. Based on the error estimates,
		we prove that the numerical solution can also blow up under large initial value.
		There are several numerical examples in \cite{2020Unconditionally} to validate the blowup behavior of the numerical solution to the fully discrete scheme \eqref{eq:fem1}-\eqref{eq:fem2}. The analysis of  Theorem \ref{th:blow} depends on the regularity of the solution, it is very interesting whether we can still have similar results under weak regularity, we will continue to conduct on this issue in the future. 
	\end{remark}
	\section{Conclusion}\label{sec6}
	In this paper, we established error estimates for a  fully discrete scheme proposed in \cite{2020Unconditionally} for the classical parabolic-elliptic Keller-Segel system, and showed that the numerical solution will blow up in finite time under some assumptions, similar to the situation for the exact solution of the classical parabolic-elliptic Keller-Segel system. 
	
	\section*{Acknowledgments}
	W. Chen is supported by the National Natural Science Foundation of China (NSFC) 12071090 and the work of J. Shen is partially supported by NSF Grant DMS-2012585 and AFOSR Grant FA9550-20-1-0309.
	
\begin{appendices}
	\section{Appendix}
	\begin{lemma}\label{lem:A_h}
		Denote $A=\alpha(-\Delta+I)^{-1},A_h=\alpha(-\Delta_h+I)^{-1}Q_h$, 
		where $Q_h$ is defined as $(Q_hu,\psi_h)=(u,\psi_h)_h$,
	then the following estimate holds for all $u\in H^2(\Omega)$, 
		\begin{align}
		\|A_hu-Au\|_{L^2(\Omega)}+h\|A_hu-Au\|_{H^1(\Omega)}\le Ch^2\|u\|_{H^2(\Omega)}.
		\end{align}
	\end{lemma} 
	\begin{proof}
		Let $w=Au$ and $w_h=A_hu$, we have the following equations:
		\begin{align*}
			(\nabla w,\nabla\psi)+(w,\psi)=\alpha(u,\psi),\ \forall\psi\in H^1,\\
			 (\nabla w_h,\nabla\psi)+(w_h,\psi)=\alpha(u,\psi)_h,\ \forall\psi\in X_h.
		\end{align*}
		Then we have
		\begin{align*}
		(\nabla (w_h-w),\nabla\psi)+(w_h-w,\psi)=\alpha((u,\psi)_h-(u,\psi)).
		\end{align*}
		Taking $\psi=w_h-\psi_h$ and denoting $w_h-w$ by $w_h-\psi_h-(w-\psi_h)$ in the above equation leads to
		\begin{align*}
			\|w_h-\psi_h\|_{H^1}^2\le 	\|w-\psi_h\|_{H^1}\|w_h-\psi_h\|_{H^1}+\alpha((u,w_h-\psi_h)_h-(u,w_h-\psi_h)).
		\end{align*}
			From \cite{2006Thomee} and
			Lemma \ref{lem:qua}, we have the following estimate
		\begin{align*}
		&|(u,w_h-\psi_h)-(u,w_h-\psi_h)_h|\\
		=&|(u-I_hu,w_h-\psi_h)-(u-I_h,w_h-\psi_h)_h+(I_hu,w_h-\psi_h)-(I_hu,w_h-\psi_h)_h|\\
		\le& Ch|u|_{H^1}\|w_h-\psi_h\|_{L^2}+Ch^2\|\nabla u\|_{L^2}\|\nabla(w_h-\psi_h)\|_{L^2}\\
		\le& Ch|u|_{H^1}\|w_h-\psi_h\|_{H^1}.
		\end{align*}
		Combing the above estimates with the elliptic regularity estimate leads to 
		\begin{align*}
			\|w-w_h\|_{H^1}&\le C\inf_{\forall \psi_h\in X_h}\|w-\psi_h\|_{H^1}+Ch|u|_{H^1}\\
			&\le Ch|w|_{H^2}+Ch|u|_{H^1}\le Ch\|u\|_{H^1}.
					\end{align*}
		 The $L^2$ error estimate can be obtained by using the duality argument
		\begin{align*}
			\|w-w_h\|_{L^2}\le Ch\|w-w_h\|_{H^1}+Ch^2|u|_{H^2}\le Ch^2\|u\|_{H^2}.
		\end{align*}
		Combing above estimates with the definitions of $A$ and $A_h$ shows that
		\begin{align*}
		&\|A_hu-Au\|_{L^2(\Omega)}+h\|A_hu-Au\|_{H^1(\Omega)}\\
		=&\|w_h-w\|_{L^2(\Omega)}+h\|w_h-w\|_{H^1(\Omega)}\\
		\le&  Ch^2\|u\|_{H^2(\Omega)},
		\end{align*}
		which completes the proof.
	\end{proof}
 In order to obtain Lemma \ref{lem:f3}, we proceed in several steps. Firstly, we deal with $\hat{u}_1:=u+hf_1$ and construct $f_1$ as follows.
        \begin{lemma}\label{lem:f1}
		Assume that $\tau\le Ch$ and $u$ is smooth enough, then there exists a bounded smooth  function $f_1$, such that $\hat{u}_1:=u+hf_1$ satisfies
		\begin{align}\label{hatu1}
		\left(\overline{\partial}\hat{u}_1^k,\varphi\right)_h+\left(\hat{u}_1^k\nabla I_h\log \hat{u}_1^{k+1},\nabla\varphi\right)_h-\chi\left(\hat{u}_1^k\nabla A_h\hat{u}_1^k,\nabla\varphi\right)_h=\langle\mathcal{R}_1^k,\varphi\rangle,
		\end{align}
		for all $\varphi\in X_h$, $k\in\mathbb{N}$, where $A_h$ is defined as in Lemma \ref{lem:A_h} and
		\begin{align*}
		|\langle\mathcal{R}^k_1,\varphi\rangle|\le Ch^2\|\varphi\|_{H^1}.
		\end{align*}
	\end{lemma}
\begin{proof}
	From \eqref{hatu1}, we have the following equality:
	\begin{equation}\label{eq:r1k}
		\begin{aligned}
		&(\dfrac{\partial\hat{u}_1}{\partial t},\varphi)+(\nabla\hat{u}_1,\nabla\varphi)-\chi(\hat{u}_1\nabla A\hat{u}_1,\nabla\varphi)
		\\=&(\dfrac{\partial\hat{u}_1}{\partial t},\varphi)+(\nabla\hat{u}_1,\nabla\varphi)-\chi(\hat{u}_1\nabla A\hat{u}_1,\nabla\varphi)+\langle\mathcal{R}^k_1,\varphi\rangle\\
		&-\left(\overline{\partial}\hat{u}_1^k,\varphi\right)_h-\left(\hat{u}_1^k\nabla I_h\log \hat{u}_1^{k+1},\nabla\varphi\right)_h+\chi\left(\hat{u}_1^k\nabla A_h\hat{u}_1^k,\nabla\varphi\right)_h,
		\end{aligned}
	\end{equation}
	where the operators $A$ and $A_h$ are defined as in Lemma \ref{lem:A_h}. Moreover, 	
	the left hand of \eqref{eq:r1k} can be rewritten as
	\begin{equation*}\label{t0}
	\begin{aligned}
	&(\dfrac{\partial\hat{u}_1}{\partial t},\varphi)+(\nabla\hat{u}_1,\nabla\varphi)-\chi(\hat{u}_1A\hat{u}_1,\nabla\varphi)\\
	=&h((\dfrac{\partial f_1}{\partial t},\varphi)+(\nabla f_1,\nabla\varphi)-\chi(uAf_1,\nabla\varphi)-\chi(f_1Au,\nabla\varphi))\\
	&-h^2(f_1Af_1,\nabla\varphi),
	\end{aligned}
	\end{equation*}
	where equation \eqref{eq:weak01} has been used.
	
{ \it Step 1: Construction for $f_1$.}  For any $\varphi\in H^1$, define $\langle\mathcal{R}^k_0(u),\varphi\rangle$ as
\begin{align*}
\langle\mathcal{R}^k_0(u),\varphi\rangle:=&(\dfrac{\partial u}{\partial t}-\overline{\partial} u^k,\varphi)+(\nabla(u-u^{k+1}),\nabla\varphi)-(u^k\nabla (I_h-I)(\log u^{k+1}),\nabla\varphi)_h\\
&-((u^k-u^{k+1})\nabla\log u^{k+1},\nabla\varphi)_h+(\nabla u^{k+1},\nabla\varphi)-(\nabla u^{k+1},\nabla\varphi)_h\\
&+\chi(u^k\nabla A_hu^k-u\nabla A_hu,\nabla\varphi)_h+\chi(u\nabla(A_h-A)u,\nabla\varphi)\\
:=&\langle\mathcal{N}(u),\varphi\rangle+\chi(u^k\nabla A_hu^k-u\nabla A_hu,\nabla\varphi)_h+\chi(u\nabla(A_h-A)u,\nabla\varphi),
\end{align*}
 we can show that $\frac{1}{h}\mathcal{R}^k_0(u)\in H^{-1}$ is well defined. 
Using the Cauchy-Schwarz inequality and the property of the interpolation, we obtain
\begin{align*}
|(\nabla u^{k+1},\nabla\varphi)-(\nabla u^{k+1},\nabla\varphi)_h|
=&(\nabla (I-I_h)u^{k+1},\nabla\varphi)-(\nabla (I-I_h)u^{k+1},\nabla\varphi)_h\\
\le& Ch\|u^{k+1}\|_{H^2}\|\nabla\varphi\|_{L^2},
\end{align*}
then the following estimate holds for $\langle\mathcal{N}(u),\varphi\rangle$:
\begin{align}\label{eq:Nu}
\langle\mathcal{N}(u),\varphi\rangle\le C_4h\|\varphi\|_{H^1},
\end{align}
and the positive constant
 $$C_4=C_4(\|u\|_{W^{2,\infty}(0,T;H^2)}+\|u\|_{W^{1,\infty}(0,T;L^{\infty})}\|\log u\|_{L^{\infty}(0,T;H^2)}).$$ 
An application of Lemma \ref{lem:A_h} leads to
\begin{align*}
|(u\nabla(A_h-A)u,\nabla\varphi)|
\le& \|u\|_{L^{\infty}}\|\nabla(A_hu-Au)\|_{L^2}\|\nabla\varphi\|_{L^2}\\
\le& Ch\|u\|_{L^{\infty}}\|u\|_{H^2}\|\nabla\varphi\|_{L^2}.
\end{align*}
Combining \eqref{eq:Nu} with few direct calculations shows that the following estimate holds for $\langle\mathcal{R}^k_0(u),\varphi\rangle$:
\begin{align*}
|\langle\mathcal{R}^k_0(u),\varphi\rangle|\le C_5h\|\varphi\|_{H^1},
\end{align*}
where $$C_5=C_5(\|u\|_{W^{2,\infty}(0,T;H^2)}+\|u\|_{W^{1,\infty}(0,T;L^{\infty})}\|\log u\|_{L^{\infty}(0,T;H^2)}+\|u\|_{W^{1,\infty}(0,T;L^{\infty})}^2),$$ then $\frac{1}{h}\langle\mathcal{R}^k_0(u),\varphi\rangle$ is well-defined. 
Combining $\langle\mathcal{R}^k_0(u),\varphi\rangle$ with \eqref{eq:r1k} leads to the following linear partial differential equation for $f_1$:
\begin{align}\label{eq:f1}
(\dfrac{\partial f_1}{\partial t},\varphi)+(\nabla f_1,\nabla\varphi)-\chi(u\nabla Af_1,\nabla\varphi)-\chi(f_1\nabla Au,\nabla\varphi)=\dfrac{1}{h}\langle\mathcal{R}^k_0(u),\varphi\rangle,
\end{align}
for all $t\in(t_k,t_{k+1}]$. From \cite[Chapter7.1, Theorem 3]{evans}, there exists a weak solution $f_1$ of \eqref{eq:f1}. In addition, from \cite[Chapter7.1, Theorem 7]{evans}, suppose that $u$ is smooth enough such that $\mathcal{R}_0^k(u)$ is smooth enough in $\overline{\Omega}\times[t_k,t_{k+1}]$, and the $m^{th}$-order compatibility conditions hold for $m=0,1,\cdots$, then the problem \eqref{eq:f1} has a smooth enough solution $f_1$ in $\overline{\Omega}\times[t_k,t_{k+1}]$.

{\it Step 2: Construction for $\langle\mathcal{R}^k_1(\hat{u}_1),\varphi\rangle$.} 	Let $\langle\mathcal{R}^k_1(\hat{u}_1),\varphi\rangle$ be
	\begin{align*}
	\langle\mathcal{R}^k_1(\hat{u}_1),\varphi\rangle
	:=&h\langle\mathcal{N}(f_1),\varphi\rangle+\chi h((u\nabla(A_h-A)f_1+f_1\nabla(A_h-A)u,\nabla\varphi))\\
	+&\chi h((u^k\nabla A_hf_1^k-u\nabla A_hf_1,\nabla\varphi)_h+(f_1^k\nabla A_hu^k-f_1\nabla A_hu,\nabla\varphi)_h)\\
	-&(u^k\nabla I_h(\log (u^{k+1}+hf_1^{k+1})-\log u^{k+1}-\dfrac{hf_1^{k+1}}{u^{k+1}}),\nabla\varphi)_h\\
	-&h(u^k\nabla(I_h-I)\dfrac{f_1^{k+1}}{u^{k+1}},\nabla\varphi)_h-h((u^k-u^{k+1})\nabla\dfrac{f_1^{k+1}}{u^{k+1}},\nabla\varphi)_h\\
	+&\left(\overline{\partial} u^k,\varphi\right)-\left(\overline{\partial} u^k,\varphi\right)_h+h^2(f_1^k\nabla\dfrac{f_1^{k+1}}{u^{k+1}},\nabla\varphi)+\chi h^2(f_1\nabla Af_1,\nabla\varphi)\\
	+&\chi((u\nabla A_hu,\nabla\varphi)_h-(u\nabla A_hu,\nabla\varphi)).
	\end{align*}
	Similarly, the following estimate holds for $ \langle\mathcal{N}(f_1),\varphi\rangle$ as discussed in \eqref{eq:Nu}:
	\begin{align}\label{eq:Nf1}
		\langle\mathcal{N}(f_1),\varphi\rangle\le C_6h\|\varphi\|_{H^1},
	\end{align} 
	where the positive constant $$C_6=C_6(\|f_1\|_{W^{2,\infty}(0,T;H^2)}+\|f_1\|_{W^{1,\infty}(0,T;L^{\infty})}\|\log u\|_{L^{\infty}(0,T;H^2)}).$$
	Combining \eqref{eq:Nf1} with few calculations yields the following estimate for  $\langle\mathcal{R}^k_1(\hat{u}_1),\varphi\rangle$:	
	\begin{align*}
	|\langle\mathcal{R}^k_1(\hat{u}_1),\varphi\rangle|\le C_7h^2\|\varphi\|_{H^1},
	\end{align*}
	where the positive constant
	\begin{align*}
		C_7=C_7(\|u\|_{W^{1,\infty}(0,T;H^2)}+(\|u\|_{W^{1,\infty}(0,T;L^{\infty})}+\|f_1\|_{L^{\infty}(0,T;L^{\infty})})\|\frac{f_1}{u}\|_{L^{\infty}(0,T;H^2)}\\
		+\|u\|_{W^{1,\infty}(0,T;L^{\infty})}\|f_1\|_{W^{1,\infty}(0,T;L^{\infty})}+\|u\|_{L^{\infty}(0,T;H^2)}^2+\|f_1\|_{L^{\infty}(0,T;L^{\infty})}^2)+C_6,
	\end{align*}
	 then the $O(h^2)$ consistency for $\hat{u}_1=u+hf_1$ is obtained, which leads to Lemma \ref{lem:f1}.
\end{proof}
\begin{remark}
	Taking $\varphi=1$ in \eqref{hatu1} leads to  $(\hat{u}_1^k,1)_h=(\hat{u}_1^0,1)_h$ for $k\in\mathbb{N}$, i.e., $\hat{u}_1$ preserves mass conservation property. Choosing suitable initial condition  $f_1(x,0)=0$ such that $(f_1(x,0),1)_h=0$, we obtain $(\hat{u}_1^k,1)_h=(u^0,1)_h$.
\end{remark}

After repeated application of the perturbation argument as illustrated in Lemma \ref{lem:f1}, Lemma \ref{lem:f3} can be proved.

	\begin{proof}[Proof of Lemma \ref{lem:f3}] 
	The duality product $\frac{1}{h^2}\langle\mathcal{R}^k_1(\hat{u}_1),\varphi\rangle$ is well-defined since the fact that $\frac{1}{h^2}\langle\mathcal{R}^k_1(\hat{u}_1),\varphi\rangle$ is uniformly bounded as $h\rightarrow0,\tau\rightarrow0$ and $\tau\le Ch$. We can construct $f_2$ by solving the following linear partial differential equation:
	\begin{align}\label{eq:f2}
	(\dfrac{\partial f_2}{\partial t},\varphi)+(\nabla f_2,\nabla\varphi)-\chi(u\nabla Af_2,\nabla\varphi)-\chi(f_2\nabla Au,\nabla\varphi)=\dfrac{1}{h^2}\langle\mathcal{R}^k_1(\hat{u}_1),\varphi\rangle,
	\end{align}
	for all $t\in(t_k,t_{k+1}]$. As discussed in Step 1 above, the problem \eqref{eq:f2} has a smooth enough solution $f_2$ in $\overline{\Omega}\times[t_k,t_{k+1}]$.

	By repeated application of the methods in Step 2 above, we can construct $\langle\mathcal{R}^k_2(\hat{u}_2),\varphi\rangle $ by $\hat{u}_2:=\hat{u}_1+h^2f_2$  such that the $O(h^3)$ consistency for $\hat{u}_2$ is arrived: 
	\begin{align}\label{eq:r2k}
	\left(\overline{\partial}\hat{u}_2^k,\varphi\right)_h+\left(\hat{u}_2^k\nabla I_h\log \hat{u}_2^{k+1},\nabla\varphi\right)_h-\chi\left(\hat{u}_2^k\nabla A_h\hat{u}_2^k,\nabla\varphi\right)_h=\langle\mathcal{R}^k_2(\hat{u}_2),\varphi\rangle,
	\end{align}
	for all $\varphi\in X_h$, $k\in\mathbb{N}$, where
	\begin{align*}
	|\langle\mathcal{R}^k_2(\hat{u}_2),\varphi\rangle|\le Ch^3\|\varphi\|_{H^1},
	\end{align*}
	where $C$ is a positive constant depending on the derivatives of $\hat{u}_2$, such that $\frac{1}{h^3}\langle\mathcal{R}^k_2(\hat{u}_2),\varphi\rangle$ is well-defined.
	
	Again, by using Step 1 in Lemma \ref{lem:f1}, the correction function $f_3$ can be constructed by the following linear partial differential equation:
	\begin{align}\label{eq:f3}
	(\dfrac{\partial f_3}{\partial t},\varphi)+(\nabla f_3,\nabla\varphi)-\chi(u\nabla Af_3,\nabla\varphi)-\chi(f_3\nabla Au,\nabla\varphi)=\dfrac{1}{h^3}\langle\mathcal{R}^k_2(\hat{u}_2),\varphi\rangle,
	\end{align}
	for all $t\in(t_k,t_{k+1}]$, and the problem \eqref{eq:f3} has a smooth enough solution $f_3$ in $\overline{\Omega}\times[t_k,t_{k+1}]$. 
	
	Combing equations \eqref{eq:f1},\eqref{eq:f2} and \eqref{eq:f3} leads to
	\begin{equation}\label{hatu}
	\begin{aligned}
	&(\dfrac{\partial\hat{u}}{\partial t},\varphi)+(\nabla\hat{u},\nabla\varphi)-\chi(\hat{u}A\hat{u},\nabla\varphi)\\
	=&\langle\mathcal{R}_0^k(u),\varphi\rangle+\langle\mathcal{R}_1^k(\hat{u}_1),\varphi\rangle+\langle\mathcal{R}_2^k(\hat{u}_2),\varphi\rangle-h^4\chi((f_2\nabla Af_2,\nabla\varphi)\\
	&+(f_1\nabla Af_3+f_3\nabla A f_1,\nabla\varphi))-h^5\chi(f_2\nabla Af_3+f_3\nabla Af_2,\nabla\varphi)\\
	&-h^6\chi(f_3\nabla Af_3,\nabla\varphi).
	\end{aligned}
	\end{equation}
	Denote $\langle\mathcal{R}^k(\hat{u}),\varphi\rangle$ as follows
	\begin{align*}
	\langle\mathcal{R}^k(\hat{u}),\varphi\rangle
	&	:=h^3((\overline{\partial}f_3^k-\dfrac{\partial f_3}{\partial t},\varphi)+(\overline{\partial}f_3^k,\varphi)_h-(\overline{\partial}f_3^k,\varphi)\\
	&+(f_3^k\nabla(I_h-I)\log \hat{u}_2^{k+1},\varphi)_h-((u^{k+1}-u^k)\nabla\dfrac{f_3^{k+1}}{\hat{u}_2^{k+1}},\nabla\varphi)_h\\
	&+(\nabla f_3^{k+1},\nabla\varphi)_h-(\nabla f_3^{k+1},\nabla\varphi)+((f_3^k-f_3^{k+1})\nabla\log \hat{u}_2^{k+1},\nabla\varphi)_h)\\
	&+h^6(f_3^k\nabla\dfrac{f_3^{k+1}}{\hat{u}_2^{k+1}},\nabla\varphi)_h+(\hat{u}^k\nabla(\log\hat{u}^{k+1}-\log\hat{u}_2^{k+1}-\dfrac{h^3f_3^{k+1}}{\hat{u}_2^{k+1}}),\nabla\varphi)_h\\
		&+(\hat{u}_2^k\nabla(\log\hat{u}^{k+1}-\log\hat{u}_1^{k+1}-\dfrac{h^2f_2^{k+1}}{\hat{u}_1^{k+1}}),\nabla\varphi)_h+h^4(f_2^k\nabla\dfrac{f_2^{k+1}}{\hat{u}_1^{k+1}},\nabla\varphi)_h\\
		&+h^2((\overline{\partial}f_2^k,\varphi)_h-(\overline{\partial}f_2^k,\varphi))+h^3\chi((f_3^k\nabla A_h\hat{u}_2^k+\hat{u}_2^k\nabla A_hf_3^k,\nabla\varphi)_h\\
		&-(f_3\nabla A_h\hat{u}_2+\hat{u}_2\nabla A_hf_3,\nabla\varphi))+h^6\chi((f_3^k\nabla A_hf_3^k,\nabla\varphi)_h\\
		&-(f_3\nabla A_hf_3,\nabla\varphi))+h^4\chi((f_2^k\nabla A_hf_2^k,\nabla\varphi)_h-(f_2\nabla A_hf_2,\nabla\varphi))\\
		&+h^2\chi((f_2\nabla A_h\hat{u}_1+\hat{u}_1\nabla A_hf_2,\nabla\varphi)_h-(f_2\nabla A_h\hat{u}_1+\hat{u}_1\nabla A_hf_2,\nabla\varphi)).
	\end{align*}
	Combining above with few direct calculations shows the following estimate for $\langle\mathcal{R}^k(\hat{u}),\varphi\rangle$
	\begin{align*}
		|\langle\mathcal{R}^k(\hat{u}),\varphi\rangle|\le Ch^4\|\varphi\|_{H^1},
	\end{align*}
	where $C$ depending on the derivatives of $u$, then the $\frac{1}{h^4}\langle\mathcal{R}^k(\hat{u}),\varphi\rangle$ is well defined and the $O(h^4)$ consistency holds for $\hat{u}=u+hf_1+h^2f_2+h^3f_3$, which leads to Lemma \ref{lem:f3}. 
\end{proof}
\begin{remark}\label{remak:massconserve}
	Similarly, taking $\varphi=1$ in \eqref{eq:consistency}  leads to $(\hat{u}^k,1)_h=(\hat{u}^0,1)_h$ for $k\in\mathbb{N}$. Choosing the initial condition $f_i(x,0)=0$ such that $(f_i(x,0),1)_h=0$ ($i=1$, $2$, $3$), we obtain $(\hat{u}^k,1)_h=(u^0,1)_h$.
\end{remark}

\end{appendices}


\begin{thebibliography}{99}
		
		\bibitem{Sobolev}
		R.~Adams and J.~Fournier, Sobolev Spaces, 2nd~Edition,
		Academic Press, Singapore, 2003.
		
		\bibitem{2012Blanchet}
		A.~Blanchet and P.~Laurençot, The parabolic-parabolic Keller-Segel system
		with critical diffusion as a gradient flow in $\mathbb{R}^d, d\ge 3$, Communications in Partial Differential Equations, 38(4): 658-686, 2012.
		
		\bibitem{2006Blanchet}
		A.~Blanchet, J.~Dolbeault and B.~Perthame, Two-dimensional Keller-Segel
		model: Optimal critical mass and qualitative properties of the solutions, Electronic Journal of Differential Equations, 2006(44): 285--296, 2016. 
		
		
		\bibitem{2008brenner}
		S.~C. Brenner and L.~R. Scott, The Mathematical Theory of
		Finite Element Methods, 3rd Edition, Springer, New York,
		2007.
		
		
		\bibitem{2007The}
		V.~Calvez and L.~Corrias, The parabolic-parabolic Keller-Segel model in
		$\mathbb{R}^2$, Communications in Mathematical Sciences, 6(2): 417--447, 2008.
		
		\bibitem{2019Calvez}
		V.~Calvez, L.~Corrias and M.~A.~Ebde, Blow-up, concentration phenomenon and
		global existence for the Keller-Segel model in high dimension,
		Journal of Differential Equations, 267(11): 561--584, 2019.
		
		\bibitem{2019chen}
		W.~Chen, C.~Wang, X.~Wang and S.~Wise, Positivity-preserving, energy stable numerical schemes for the Cahn-Hilliard equation with logarithmic potential, 
		Journal of Computational Physics X, 3: 100031, 2019. 
		
		\bibitem{1995Diaz}
		J.~I. Diaz and T.~Nagai, Symmetrization in a parabolic-elliptic system
		related to chemotaxis, Advances in Mathematical Sciences and Applications, 5(2): 659-680, 1995.
		
		\bibitem{2004Optimal}
		J.~Dolbeault and B.~Perthame, Optimal critical mass in the two-dimensional
		Keller-Segel model in $\mathbb{R}^2$. Retour Au Numéro, 339(9): 611--616, 2004.
		
		\bibitem{2019dong}
		L.~Dong, C.~Wang, H.~Zhang and Z.~Zhang, A positivity-preserving, energy stable and convergent numerical scheme for the Cahn-Hilliard equation with a Flory-Huggins-deGennes energy, 
		Communications in Mathematical Sciences, 17(4): 921-939, 2019.  
		
		\bibitem{2020dong}
		L.~Dong, C.~Wang, H.~Zhang and Z.~Zhang, 
		A positivity-preserving second-order BDF scheme for the Cahn-Hilliard equation with variable interfacial parameters, 
		Communications in Computational Physics, 28(3): 967-998, 2020. 
		
		\bibitem{2021dong}
		L.~Dong, C.~Wang, S.~Wise and Z.~Zhang, 
		A positivity preserving, energy stable scheme for the ternary Cahn-Hilliard system with the singular interfacial parameters, 
		Journal of Computational Physics, 442: 110451, 2021. 
		
		\bibitem{2020duan}
		C.~Duan, C.~Liu, C.~Wang and X.~Yue, 
		Convergence analysis of a numerical Scheme for the porous medium equation by an energetic variational approach, 
		Numerical Mathematics: Theory, Methods and Applications, 13(1): 63-80, 2020. 
		
		\bibitem{2021duan}
		C.~Duan, W.~Chen, C.~Liu, C.~Wang and S.~Zhou, 
		Convergence analysis of structure-preserving numerical methods for nonlinear Fokker-Planck equations with nonlocal interactions, 
		Mathematical Methods in the Applied Sciences, accepted and in press, 2021.  
		
		\bibitem{2008epshteyn}
		Y.~Epshteyn and A.~Kurganov, New interior penalty discontinuous Galerkin methods for the Keller-Segel
		Chemotaxis model, SIAM Journal on Numerical Analysis, 47(1): 386--408, 2008.
		
		\bibitem{2009epshteyn} 
		Y.~Epshteyn and  A.~Izmirlioglu, Fully discrete analysis of a discontinuous finite element method for the Keller-Segel chemotaxis model, Journal of Scientific Computing, 40(1--3): 211--256, 2009.
		
		\bibitem{evans}
		L.~C.~Evans, Partial Differential Equations. AMS Graduate studies in Mathematics. AMS: Providence, RI, 1998.
		\bibitem{2006Filbet}
		F.~Filbet, A finite volume scheme for the Patlak–Keller–Segel chemotaxis
		model, Numerische Mathematik, 104(4): 457--488, 2006.
		
		\bibitem{2015Fujie}
		K.~Fujie, M.~Winkler and T.~Yokota, Boundedness of solutions to
		parabolic-elliptic chemotaxis-growth systems with signal-dependent
		sensitivity, Mathematical Methods in the Applied Sciences, 38(6): 1212--1224, 2015.
		
		\bibitem{2021arXiv210106748F}
		J.~{Fuhrmann}, J.~{Lankeit}
		and M.~{Winkler}, A double critical mass
		phenomenon in a no-flux-Dirichlet Keller-Segel system, https://arxiv.org/abs/2101.06748, 2021.
		
		\bibitem{1998Gajewski}
		H.~Gajewski, K.~Zacharias, and K.~Gr\"oger, 
		Global behavior of a
		reaction-diffusion system modelling chemotaxis, Mathematische
		Nachrichten, 195(1): 77--114, 1998.
		
		\bibitem{li2019}
		L.~Guo, X.~Li, and Y.~Yang, Energy dissipative local discontinuous Galerkin
		methods for keller–segel chemotaxis model, Journal of Scientific
		Computing, 78: 1387–1404, 2019.
		
		\bibitem{2018gurusamy}
		A.~Gurusamy, K.~Balachandran, Finite element method for solving Keller-Segel chemotaxis system with cross-diffusion, International Journal of Dynamics and Control, 6:539--549,2018.
		
		\bibitem{2001Horstmann}
		D.~Horstmann, The nonsymmetric case of the Keller-Segel model in chemotaxis:
		Some recent results, Nonlinear Differential Equations and
		Applications, 8: 399--423, 2001.
		
		
		\bibitem{2004Horstmann}
		D.~Horstmann, From 1970 until present: The Keller-Segel model in chemo-taxis
		and its consequences, Jahresbericht der Deutschen
		Mathematiker-Vereinigung, 106(2): 51--69, 2004.
		
		\bibitem{2019BLOW}
		A.~J\"uengel and O.~Leingang, Blow-up of solutions to semi-discrete
		parabolic-elliptic Keller-Segel models, Discrete and Continuous
		Dynamical Systems, 24(9): 4755--4782, 2019.
		
		\bibitem{1970Keller}
		E.~F. Keller and L.~A. Segel, Initiation on slime mold aggregation viewed as instability. Journal of Theoretical Biology, 26(3): 399–415, 1970.
		
		\bibitem{1971Keller}
		E.~F. Keller and L.~A. Segel, Model for chemotaxis, Journal of Theoretical Biology,
		30(2): 225--234, 1971.
		
		\bibitem{Hideo2008Local}
		H.~Kozono and Y.~Sugiyama, Local existence and finite time blow-up of
		solutions in the 2-D Keller-Segel system, Journal of Evolution
		Equations, 8: 353--378, 2008.
		
		\bibitem{2017li}
		X.~Li, C. W. Shu, and Y.~Yang, Local discontinuous Galerkin method for the
		Keller-Segel chemotaxis model, Journal of Scientific Computing, 73: 943–967, 2017.
		
		\bibitem{2018liwj}
		W.~Li, W.~Chen, C.~Wang, Y.~Yan, and R.~He, A second order energy stable
		linear scheme for a thin film model without slope selection, Journal of Scientific Computing, 76: 1905–1937, 2018.
		
		\bibitem{2021wangcheng}
		C.~Liu, C.~Wang, S.~M. Wise, X.~Yue, and S.~Zhou, A positivity-preserving, energy stable and convergent numerical scheme for the Poisson-Nernst-Planck system, Mathematics of Computation, 90(331): 2071-2106, 2021.
		
		\bibitem{2021liu}
		C.~Liu, C.~Wang and Y.~Wang, 
		A structure-preserving, operator splitting scheme for reaction-diffusion equations with detailed balance, 
		Journal of Computational Physics, 436: 110253, 2021. 
		
		\bibitem{2016liu}
		J.~G.~Liu, L.~Wang and Z.~Zhou, Positivity-preserving and asymptotic
		preserving method for 2D Keller-Segal equations, Mathematics of
		Computation, 87(311): 1165--1189, 2018.
		
		\bibitem{2000Behavior}
		T.~Nagai, Behavior of solutions to a parabolic-elliptic system modelling
		chemotaxis, Journal of the Korean Mathematical Society, 37(5): 721--732, 2000.
		
		\bibitem{2001Concentration}
		T.~Nagai, T.~Senba, and T.~Suzuki, Concentration behavior of blow-up
		solutions for a simplified system of chemotaxis (variational problems and
		related topics), Rims Kokyuroku, 1181:
		140--176, 2001.
		
		\bibitem{2001Blowup}
		T.~Nagai, Blowup of nonradial solutions to parabolic–elliptic systems
		modeling chemotaxis in two-dimensional domains, Journal of
		Inequalities and Applications, 6: 37-55, 2001.
		
		\bibitem{2001Global}
		T.~Nagai, Global existence and blowup of solutions to a chemotaxis system,
		Nonlinear Analysis: Theory, Methods and Applications, 47(2): 777--787, 2001.
		
		\bibitem{1997Application}
		T.~Nagai, T.~Senba, and K.~Yoshida, Application of the trudinger-moser
		inequality to a parabolic system of chemotaxis, Funkc Ekvacioj,
		40(3): 411--433, 1997.
		
		\bibitem{2007Lower}
		L.~E.~Payne and P.~W.~Schaefer, Lower bounds for blow-up time in parabolic
		problems under dirichlet conditions, Journal of Mathematical Analysis
		and Applications, 328: 1196--1205, 2007.
		
		\bibitem{2021qian}
		Y.~Qian, C.~Wang and S.~Zhou,     
		A positive and energy stable numerical scheme for the Poisson-Nernst-Planck-Cahn-Hilliard equations with steric interactions, 
		Journal of Computational Physics, 426: 109908, 2021. 
		
		\bibitem{2005SaitoNotes}
		N.~Saito and T.~Suzuki, Notes on finite difference schemes to a
		parabolic-elliptic system modelling chemotaxis, Applied Mathematics
		and Computation, 171(1): 72-90, 2005.
		
		\bibitem{2007NorikazuConservative}
		N.~Saito, Conservative upwind finite-element method for a simplified Keller-Segel system modelling chemotaxis, IMA Journal of Numerical Analysis, 27(2): 332–365, 2007.
		
		\bibitem{2011SaitoERRORAO}
		N.~Saito, Error analysis of a conservative finite-element approximation for the Keller-Segel system of
		chemotaxis, Communications on Pure and Applied Analysis, 11(1): 339–364, 2012.
		
		\bibitem{1990shen}
		J.~Shen, Long time stability and convergence for fully discrete nonlinear
		Galerkin methods, Applicable Analysis, 38:  201--229,
		1990.
		
		\bibitem{2020Unconditionally}
		J.~Shen and J.~Xu, Unconditionally bound preserving and energy dissipative
		schemes for a class of Keller--Segel equations, SIAM Journal on
		Numerical Analysis, 58(3): 1674--1695, 2020.
		
		\bibitem{2005Free}
		T.~Suzuki, Free Energy and Self-Interacting Particles, Progr. Nonlinear Differential Equations Appl, 1st Edition, Birkhäuser Boston, 2005.
		
		\bibitem{2006Thomee}
		V.~Thomee, Galerkin Finite Element Methods for Parabolic Problems,
		2nd~Edition, Springer-Verlag, 2006.
		
		\bibitem{wei2018}
		D.~Wei, Global well-posedness and blow-up for the 2-D Patlak-Keller-Segel equation, Journal of Functional Analysis, 274: 388--401, 2018.
		
		\bibitem{Yan2018ASE}
		Y.~Yan, W.~Chen, C.~Wang and S.~Wise, A second-order energy stable BDF
		numerical scheme for the Cahn-Hilliard equation, Communications in Computational Physics, 23(2): 572--602, 2018.
		
		\bibitem{2021yuan}
		M.~Yuan, W.~Chen, C.~Wang, S.~Wise and Z.~Zhang, 
		An energy stable finite element scheme for the three-component Cahn-Hilliard-type model for macromolecular microsphere composite hydrogels, Journal of Scientific Computing, 87:78, 2021. 
		
		\bibitem{2016SaitoFinite}
		G.~Zhou and N.~Saito, Finite volume methods for a Keller–Segel system:
		discrete energy, error estimates and numerical blow-up analysis,
		Numerische Mathematik, 135: 1-47, 2016.
		
	\end{thebibliography}
\end{document}